\documentclass[10pt]{article}

\usepackage{amssymb}
\usepackage{dsfont}
\usepackage{amsmath}
\usepackage{graphicx}
\usepackage{caption}
\usepackage{subcaption}
\usepackage{caption}
\usepackage{float}
\usepackage{subcaption}
\usepackage{dcolumn}
\usepackage{amsthm}
\usepackage{multirow}
\usepackage{enumitem}
\usepackage{amsfonts}
\usepackage{color}
\usepackage{url}

\usepackage[sort]{cite}

\numberwithin{equation}{section}

\def\Line(#1,#2)(#3,#4){\qbezier(#1,#2)(#1,#2)(#3,#4)}

\newcommand{\inp}[2][]{\left(#1,\, #2\right)}
\newcommand{\gnp}[2][]{\langle#1,\, #2\rangle}
\newcommand{\bs}[1]{\boldsymbol{#1}}

\newtheorem{remark}{Remark}[section]
\newtheorem{lemma}{Lemma}[section]
\newtheorem{definition}{Definition}[section]

\newtheorem{theorem}{Theorem}[section]
\newtheorem{proposition}{Proposition}[section]

\def\blue{\color{black}}

\def\Tc{\mathcal{T}}
\def\Qc{\mathcal{Q}}
\def\Pc{\mathcal{P}}
\def\Kc{\mathcal{K}}
\def\Ic{\mathcal{I}}

\def\Q{\mathcal{Q}}

\def\B{\mathcal{B}}
\def\Bt{\mathcal{\tilde{B}}}
\def\Rc{\mathcal{R}}

\def\R{\mathbb{R}}

\def\s{\sigma}

\def\Eh{\hat{E}}
\def\Pih{\hat{\Pi}}
\def\eh{\hat{e}}

\def\xh{\hat{x}}
\def\yh{\hat{y}}

\def\zh{\hat{z}}
\def\Wh{\hat{W}}
\def\wh{\hat{w}}
\def\nablah{\hat{\nabla}}

\def\Bbt{\boldsymbol{\Bt}}
\def\Bb{\boldsymbol{\B}}
\def\u{\mathbf{u}}
\def\n{\mathbf{n}}
\def\r{\mathbf{r}}
\def\nh{\mathbf{\hat{n}}}
\def\rh{\mathbf{\hat{r}}}
\def\V{\mathbf{V}}
\def\Vh{\mathbf{\hat{V}}}
\def\v{\mathbf{v}}
\def\vh{\mathbf{\hat{v}}}
\def\w{\mathbf{w}}

\def\q{\mathbf{q}}
\def\qh{\mathbf{\hat{q}}}
\def\p{\mathbf{p}}
\def\ph{\mathbf{\hat{p}}}
\def\x{\mathbf{x}}
\def\xbh{\mathbf{\xh}}

\def\Ki{K^{-1}}
\def\Kci{\Kc^{-1}}

\def\dvrg{\nabla\cdot}   
\def\dvr{\operatorname{div}}  

\def\O{\Omega}
\def\dO{\partial\Omega}
\def\G{\Gamma}
\def\Gn{\Gamma_N}
\def\Gd{\Gamma_D}

\usepackage{geometry}
\begin{document}
	
\title{Higher order multipoint flux mixed finite element methods on 
quadrilaterals and hexahedra}

\author{Ilona Ambartsumyan\thanks{Department of Mathematics, University of
Pittsburgh, Pittsburgh, PA 15260, USA;~{\tt \{ila6@pitt.edu, elk58@pitt.edu, yotov@math.pitt.edu\}}. Partially supported by DOE grant DE-FG02-04ER25618 and NSF grants DMS 1418947 and DMS 1818775.} \thanks{Institute for Computational Engineering and Sciences, The University of Texas at Austin, Austin, TX 78712, USA;
{\tt \{ailona@austin.utexas.edu, ekhattatov@austin.utexas.edu\}}.}~\and
Eldar Khattatov\footnotemark[1] \footnotemark[2] \and
		Jeonghun Lee\footnotemark[3]\thanks{Department of Mathematics, Baylor University, Waco, TX 76798, USA;~{\tt \{Lee\_Jeonghun@baylor.edu\}}.}\and
		Ivan Yotov\footnotemark[1]~}

	\date{\today}
	\maketitle
	\begin{abstract}
	We develop higher order multipoint flux mixed finite element
        (MFMFE) methods for solving elliptic problems on quadrilateral
        and hexahedral grids that reduce to cell-based pressure
        systems. The methods are based on a new family of mixed finite
        elements, which are enhanced Raviart-Thomas spaces with
        bubbles that are curls of specially chosen polynomials. The
        velocity degrees of freedom of the new spaces can be
        associated with the points of tensor-product Gauss-Lobatto
        quadrature rules, which allows for local velocity elimination
        and leads to a symmetric and positive definite cell-based
        system for the pressures. We prove optimal $k$-th order
        convergence for the velocity and pressure in their natural
        norms, as well as $(k+1)$-st order superconvergence for the
        pressure at the Gauss points.  Moreover, local
        postprocessing gives a pressure that is superconvergent of
        order $(k+1)$ in the full $L^2$-norm. Numerical results
        illustrating the validity of our theoretical results are
        included.
	\end{abstract}


	
\section{Introduction.}	
Mixed finite element (MFE) methods
\cite{brezzi1991mixed,roberts1991mixed} are commonly used for modeling
of fluid flow and transport, as they provide accurate and locally mass
conservative velocities and robustness with respect to heterogeneous,
anisotropic, and discontinuous coefficients. A disadvantage of the MFE
methods in their standard form is that they result in coupled
velocity-pressure algebraic systems of saddle-point type, which
restricts the use of efficient iterative solvers. To address this
issue, there has been extensive work on developing modifications of
MFE methods that can be reduced to positive definite systems, such as
hybridization \cite{arnold1985mixed,brezzi1991mixed} or relating them
to cell-centered finite difference or finite volume methods.  In the
latter approach, a common technique is to employ special quadrature
rules, also referred to as mass lumping, that allow for local velocity
elimination, resulting in cell-centered pressure systems,
Early works
\cite{russell1983finite,weiser1988convergence,baranger1996connection}
based on the lowest order Raviart-Thomas (RT$_0$) spaces
\cite{raviart1977mixed} were limited to two-point flux
approximations, which were not robust for general quadrilateral grids
or tensor-valued coefficients. An extension to higher order RT spaces,
as well as the second order Brezzi-Douglas-Fortin-Marini (BDFM$_2$)
spaces \cite{brezzi1991mixed} was developed in
\cite{cai2003development}, but was also limited to rectangular grids
and diagonal tensor coefficients.  The expanded MFE method
\cite{arbogast1998enhanced,arbogast1997mixed} was designed to handle
full tensor coefficients and general grids, but suffered from reduced
convergence for problems with discontinuous coefficients.

More recently, a special MFE method, the multipoint flux mixed finite
element (MFMFE) method
\cite{wheeler2006multipoint,ingram2010multipoint} was developed, which
reduces to cell-centered finite differences on quadrilateral,
hexahedral and simplicial grids, and exhibits robust performance for
discontinuous full tensor coefficients. The method was motivated by
the multipoint flux approximation (MPFA) method
\cite{aavatsmark2002introduction,aavatsmark1998discretization,
  edwards2002unstructured,edwards1998finite,
  aavatsmark2007convergence}, which was developed as a finite volume
method. Unlike the MPFA method, the variational formulation of the
MFMFE method allows for its complete theoretical study of
well-posedness and convergence. The MFMFE method is based on the
lowest order Brezzi-Douglas-Marini (BDM$_1$) space
\cite{brezzi1985two,Nedelec86} on simplices and quadrilaterals, and an enhanced
Brezzi-Douglas-Duran-Fortin (BDDF$_1$) space
\cite{BDDF,ingram2010multipoint} on hexahedra. The method utilizes the
trapezoidal quadrature rule for the velocity mass matrix, which
reduces it to a block-diagonal form with blocks associated with mesh
vertices. The velocities can then be easily eliminated, resulting in a
cell-centered pressure system. A similar approach was also presented
independently in \cite{Brezzi-Fortin-Marini-2006} for simplicial
grids, and a related formulation based on a broken Raviart-Thomas
space was developed in
\cite{Klausen-Winther-2006a,Klausen-Winther-2006b}. Motivated by the
work in \cite{Klausen-Winther-2006a}, a nonsymmetric version of the
MFMFE method designed to converge on general quadrilateral and
hexahedral grids was developed in \cite{WheXueYot-nonsym}. A
multiscale mortar MFMFE method on multiple subdomains with
non-matching grids was proposed in \cite{WheXueYot-msmortar}. In
\cite{LSY}, a local flux mimetic finite difference method was
developed on polyhedral grids, exploring connections to the MFMFE and
MPFA methods, see also related work in
\cite{Klaus-Steph,Steph}. Furthermore, on simplicial grids and for
problems with full tensor coefficients, using the MPFA principle, it
was shown in \cite{Vohralik-mpfa,Vohralik-Wohlmuth,Younes-Fontaine}
that the RT$_0$ MFE method can be related to a finite volume method
with one pressure unknown per element.

To the best of our knowledge, the aforementioned 
MPFA and MFMFE methods with theoretical 
convergence proofs are limited to the lowest order
approximation. In this paper we develop a family of arbitrary order
symmetric MFMFE methods on $h^2$-perturbed quadrilateral and hexahedral grids.  The
main obstacle in extending the original lowest order BDM$_1$ and
BDDF$_1$ MFMFE methods to higher order is that the degrees of freedom
of their higher order versions cannot be associated with
tensor-product quadrature rules. To circumvent this difficulty, we
construct a new family of mixed finite elements fulfilling this
requirement.  A key of the construction is the finite element exterior
calculus framework \cite{AFW06,AFW10}, which is used in the extension of MFMFE to Hodge
Laplace equations \cite{Lee-Winther}. However, we consider only the two and
three dimensional cases with $H(\dvr)$ element in this paper, 
so no prerequisite of the exterior calculus language is necessary.  
The new spaces are enhanced Raviart-Thomas spaces with bubbles that are curls of specially chosen
polynomials, so that each component of the velocity vector is of
dimension $\Qc^k(\R^d)$ and the velocity degrees of freedom can be
associated with the points of a tensor-product Gauss-Lobatto
quadrature rule \cite{abramowitz1964handbook}. The application of this
quadrature rule leads to a block-diagonal velocity mass matrix with
blocks corresponding to the nodes associated with the velocity degrees
of freedom. This allows for a local elimination of the fluxes in terms
of the pressures from the surrounding elements, either sharing a
vertex, or an edge/face.  This procedure results in a symmetric and
positive-definite cell-based system for the pressures with a compact
stencil, allowing for efficient solvers to be used. The
proposed technique allows for more straightforward and efficient
implementation and results in reduced computational time. 
We remark that the lowest order version of our new elements has
the same number of degrees of freedom as the elements used in previous 
MFMFE methods, but they are {\it different} elements. This work
is {\it not} 
a direct extension of the previous MFMFE methods to higher order, but 
a new framework for explicit construction of higher order MFMFE methods.

We present well-posedness and convergence analysis of the proposed
family of higher order methods. To this end, we establish unisolvency
and approximation properties of arbitrary order $k$ of the new family
of enhanced Raviart-Thomas family of spaces. Since we study the
symmetric version of the MFMFE method, which relies on mapping to a
reference element via the Piola transformation, the analysis is
limited to $h^2$-perturbed parallelograms or parallelepipeds, similar
to the restriction in the lowest order symmetric MFMFE method
\cite{wheeler2006multipoint,ingram2010multipoint}. The convergence
analysis combines MFE analysis tools with quadrature error analysis,
using that the Gauss-Lobatto quadrature rule possesses sufficient
accuracy to preserve the order of convergence. We establish
convergence of $k$-th order for the velocity in the $H(\dvr)$-norm and
the pressure in the $L^2$-norm. We also employ a duality argument to
show that the numerical pressure is $(k+1)$-st order superconvergent
to the $L^2$-projection of the pressure in the finite element space,
which implies superconvergence at the Gauss points. Moreover, we
show that a variant of the local postprocessing developed in
\cite{Stenberg-91} results in a pressure that is $(k+1)$-st order
accurate in the full $L^2$-norm. All theoretical results are verified
numerically. 

The rest of the paper is organized as follows. The new family of finite
element spaces and the general order MFMFE methods are developed in
Section 2.  The error analyses for the velocity and pressure are
presented in Sections 3 and 4, respectively. Numerical experiments are
presented in Section 5.

	
\section{Definition of the method.}
\subsection{Preliminaries.} 
We consider a second order elliptic PDE written as a system of two
first order equations,
\begin{gather}
\u = -K\nabla p, \quad \dvrg \u  = f \mbox{ in } \O,\label{cts-1}\\
p = g  \mbox{ on } \Gd, \quad 
\u\cdot\n = 0 \mbox{ on } \Gn \label{cts-2},
\end{gather}
where $\O\subset\R^d$ ($d= 2, 3$) is an open bounded polytopal domain with a
boundary $\dO = \bar{\G}_D \cup \bar{\G}_N$ such that $\Gd\cap\Gn = \emptyset$,
with measure$(\Gd) > 0$. Here $\n$ is the outward unit
normal vector field on $\dO$, and $K$ is symmetric and
uniformly positive definite tensor satisfying, for some $0 < k_0 < k_1
< \infty$,
	\begin{align}
	k_0\xi^T\xi \le \xi^T K(\x) \xi \le k_1\xi^T\xi, 
\quad \forall \x\in\O,\, \forall \xi\in\R^d. \label{perm-bounds}
	\end{align}
In applications related to modeling flow in porous media, $p$ is the
pressure, $\u$ is the Darcy velocity, and $K$ represents the
permeability tensor divided by the viscosity. The above choice of boundary
conditions is made for the sake of simplicity. More general boundary
conditions, including nonhomogeneous full Neumann ones, can also be
treated.
	
Throughout the paper we will use the following standard notation. For
a domain $G\subset\R^d$, the $L^2(G)$ inner product and norm for
scalar and vector valued functions are denoted by
$\inp[\cdot]{\cdot}_G$ and $\|\cdot\|_G$, respectively. The norms and
seminorms of the Sobolev spaces $W^{s,p}(G)$, $s\in\R$, $p \ge 1$ are
denoted by $\|\cdot\|_{s,p,G}$ and $|\cdot|_{s,p,G}$,
respectively. Conventionally, the norms and seminorms of Hilbert
spaces $H^s(G)$ are denoted by $\|\cdot\|_{s,G}$ and $|\cdot|_{s,G}$,
respectively. We omit $G$ in the subscript if $G=\O$. For a section of
the domain or element boundary $S \subset \R^{d-1}$ we write
$\gnp[\cdot]{\cdot}_S$ and $\|\cdot\|_S$ for the $L^2(S)$ inner
product (or duality pairing) and norm, respectively. For a
tensor-valued function $M$, let $\|M\|_{\alpha} =
\max_{i,j}\|M_{i,j}\|_{\alpha}$ for any norm $\|M\|_{\alpha}$. We will
also use the space
	$$ H(\dvr;\O) = \{\v\in L^2(\O,\R^d) : \dvrg \v \in L^2(\O)\} $$
	equipped with the norm
	$$ \|\v\|_{\dvr} = \left( \|\v\|^2 + \|\dvrg\v\|^2 \right)^{1/2}. $$
The weak formulation for \eqref{cts-1}--\eqref{cts-2} reads as
follows: find $(\u,p) \in \V\times W$ such that
	\begin{align}
	\inp[\Ki\u]{\v} - \inp[p]{\dvrg\v} &=  - \gnp[g]{\v\cdot\n}_{\Gd}, \quad \v\in\V, \label{weak-1}\\
	\inp[\dvrg\u]{w} &= \inp[f]{w}, \quad w\in W, \label{weak-2} 
	\end{align}
	where
	$$ \V = \{\v \in H(\dvr;\O) : \v\cdot\n = 0 \mbox{ on } \Gn\},\quad W = L^2(\O). $$
It was shown \cite{brezzi1991mixed,roberts1991mixed} that 
\eqref{weak-1} - \eqref{weak-2} has a unique solution.
	
	
	\subsection{A finite element mapping.}\label{sec:map}
Let $\Tc_h$ be a finite element partition of $\O$ consisting of
quadrilaterals in 2d or hexahedra in 3d, where $h =
\max_{E\in\Tc_h}{\mbox{diam}(E)}$. We assume $\Tc_h$ to be shape
regular and quasi-uniform \cite{edwards1998finite}. For any element
$E\in\Tc_h$ there exists a bilinear (trilinear) bijection mapping
$F_E: \Eh\to E$, where $\Eh = [-1,1]^d$ is the reference square (cube). Denote
the Jacobian matrix by $DF_E$, and let $J_E = |\det(DF_E)|$. Denote
the inverse mapping by $F_E^{-1}$, its Jacobian matrix by $DF_E^{-1}$,
and let $J_{F_E^{-1}} = |\det(DF_E^{-1})|$. For $\xbh = F_E^{-1}(\x)$ 
we have that
$$ DF_E^{-1}(\x) = (DF_E)^{-1}(\xbh), 
\qquad J_{F_E^{-1}}(\x) = \frac{1}{J_E(\xbh)}. 
$$
Denote by $\rh_i$, $i=1,\dots, 2^d$, the vertices of $\Eh$, where
$\rh_1=(0,0)^T,\,\rh_2=(1,0)^T,\,\rh_3=(1,1)^T$, and $\rh_4=(0,1)^T$
in 2d, and $\rh_1=(0,0,0)^T,\,\rh_2=(1,0,0)^T,\,\rh_3=(1,1,0)^T,\,\rh_4=(0,1,0)^T,\,\rh_5=(0,0,1)^T,\,\rh_6=(1,0,1)^T,\,\rh_7=(1,1,1)^T,$ and $\rh_8=(0,1,1)^T$ in 3d.
Let $\r_i$, $i=1,\dots, 2^d$, be the
corresponding vertices of element $E$. The outward unit normal vector fields
to the facets of $E$ and $\Eh$ are denoted by $\n_i$ and $\nh_i$,
$i=1,\dots, 2d$, respectively, where facet is a face in 3d or an edge in 2d.
The bilinear (trilinear) mapping is given by
	\begin{align}
	F_E(\rh) &= \r_1 + \r_{21}\xh + \r_{41}\yh + (\r_{34}-\r_{21})\xh\yh, \quad \mbox{in 2d},\label{map-def-2d}\\
	F_E(\rh) &= \r_1 + \r_{21}\xh + \r_{41}\yh + \r_{51}\zh + (\r_{34}-\r_{21})\xh\yh + (\r_{65}-\r_{21})\xh\zh + (\r_{85}-\r_{41})\yh\zh \nonumber\\
	&\quad+ ((\r_{21}-\r_{34}) - (\r_{65}-\r_{78}))\xh\yh\zh, \,\,\quad \mbox{in 3d}\label{map-def-3d},
	\end{align}
where $\r_{ij} = \r_i - \r_j$. For the 3d case we note that
the elements can have nonplanar faces. 
	
	Let $\hat{\phi}(\xbh)$ be defined on $\Eh$, and let $\phi = \hat{\phi}\circ F_E^{-1}$. Using the classical formula $\nabla\phi = (DF_E^{-1})^T\nablah\hat{\phi}$, it is easy to see that for any facet $e_i \subset \partial E$
	\begin{align}
	\n_i = \frac{1}{J_{e_i}}J_E(DF_E^{-1})^T\nh_i, 
\quad J_{e_i} = |J_E(DF_E^{-1})^T\nh_i|_{\R^d},\label{map}
	\end{align}
where $|\cdot|_{\R^d}$ denotes the Euclidean vector norm in $\R^d$.
	Another straightforward calculation shows that, for all element types, the mapping definitions and the shape-regularity and quasi-uniformity of the grids imply that
	\begin{align}
	\|DF_E\|_{0,\infty,\Eh} \sim h, \quad \|J_E\|_{0,\infty,\Eh}\sim h^d, \quad \|DF_E^{-1}\|_{0,\infty,E} \sim h^{-1}, \mbox{ and } \|J_{F_E^{-1}}\|_{0,\infty,E} \sim h^{-d}, \label{mapping-bounds}
	\end{align}
	where the notation $a\sim b$ means that there exist positive constants $c_0,\, c_1$ independent of $h$ such that $c_0 b \le a \le c_1 b$.
	
	
\subsection{The Raviart-Thomas mixed finite element spaces.}
Let $\Pc^k$ denote the space of polynomials of total degree $\le k$ and let
$\Q^k$ denote the space of polynomials of degree $\le k$ in each variable.
We will make use of the Raviart-Thomas spaces for the construction of the 
spaces needed for the proposed method. The RT$_k$ spaces are defined for 
$k\ge 0$ on the reference cube as
	\begin{align}
	\Vh^k_{RT} (\Eh) = \begin{pmatrix}
	\Q^k + \Q^k \xh \\
	\Q^k + \Q^k \yh \\
	\Q^k + \Q^k \zh
	\end{pmatrix}, \quad 
	\hat{W}^k(\Eh) = \Q^k(\Eh).
	\end{align}
The definition on the reference square can be obtained naturally from
the one above.  
It holds that 
	\begin{align}
	\nablah\cdot \Vh^k_{RT}(\Eh) = \hat{W}^k(\Eh) \,\, \mbox{ and } \,\,
\vh\cdot\nh_{\hat{e}} \in \Qc^k(\eh) \quad \forall \vh \in \Vh^k_{RT}(\Eh), \,
\forall \eh\subset \partial \Eh. \label{rt-prop-0}
	\end{align}
The projection operator $\Pih^k_{RT}: H^1(\Eh,\R^d) \to
\Vh^k_{RT}(\Eh)$ satisfies
	\begin{align}
	& \mbox{for } k \ge 0, \quad \gnp[(\qh - \Pih^{k}_{RT}\qh)\cdot\n_{\eh}]{\hat{p}}_{\eh} = 0, \quad \forall \hat{p} \in \Qc^k(\eh), \forall \eh\subset \partial \Eh, \label{RT-dof-1} \\
& \mbox{for } k \ge 1, \quad \inp[\Pih^{k}_{RT}\qh - \qh]{\ph}_{\Eh} = 0, \quad 
\forall\ph\in 
	\begin{cases}
	\begin{pmatrix} \Pc^{k-1}(\xh)\otimes\Pc^{k}(\yh) \\ \Pc^{k-1}(\yh)\otimes\Pc^{k}(\xh) \end{pmatrix} \quad \mbox{in 2d},\\
	\begin{pmatrix} \Pc^{k-1}(\xh)\otimes\Qc^{k}(\yh,\zh) \\ \Pc^{k-1}(\yh)\otimes\Qc^{k}(\xh,\zh) \\
	\Pc^{k-1}(\zh)\otimes\Qc^{k}(\xh,\yh) \end{pmatrix} \quad \mbox{in 3d}. 
	\end{cases}\label{RT-dof-2}
	\end{align}
	The Raviart-Thomas spaces on any quadrilateral or hexahedral element $E\in\Tc_h$ are defined via the transformations
	\begin{align}
	\v \leftrightarrow \vh : \v = \frac{1}{J_E}DF_E\vh\circ F_E^{-1}, \quad w\leftrightarrow \wh : w = \wh\circ F_E^{-1}, \label{piola}
	\end{align}
where the contravariant Piola transformation is used for the velocity
space. Under this transformation, the normal components of the
velocity vectors on the facets are preserved. In particular 
\cite{brezzi1991mixed},
\begin{align}
	\forall \vh\in\Vh^k_{RT}(\Eh), \, 
\forall \wh\in\Wh^k(\Eh), 
\quad \inp[\dvrg\v]{w}_E = \inp[\nablah\cdot\vh]{\wh}_{\Eh}
\mbox{ and } \gnp[\v\cdot\n_{e}]{w}_e = \gnp[\vh\cdot\nh_{\eh}]{\wh}_{\eh},
\label{rt-prop-3}
	\end{align}
which imply
	\begin{align}
	\v\cdot\n_e = \frac{1}{J_e}\vh\cdot\nh_{\eh}, 
\quad \dvrg\v (\x) = \left( \frac{1}{J_E}\nablah\cdot\vh \right)\circ F_E^{-1}(\x). \label{rt-prop-4}
	\end{align}
The RT$_k$ spaces on $\Tc_h$ are given by
	\begin{align}
	\begin{aligned}
	\V^k_{RT,h} &= \big\{ \v\in \V: &&\v|_E \leftrightarrow \vh,\, \vh\in\Vh_{RT}^k(\Eh),\quad E\in\Tc_h \big\}, \\
	W^k_h &= \big\{ w\in W: &&w|_E \leftrightarrow \wh,\, \wh\in\Wh^k(\Eh),\quad E\in\Tc_h \big\}. \label{rt-prop-5}
	\end{aligned}
	\end{align}
Using the Piola transformation, we define a projection operator $\Pi^k_{RT}$ from $\V\cap H^1(\O,\R^d)$ onto $\V^k_{RT,h}$ satisfying on each element
	\begin{align}
	\Pi^k_{RT}\q \leftrightarrow \widehat{\Pi^k_{RT}\q},\quad \widehat{\Pi^k_{RT}\q} = \hat{\Pi}^k_{RT}\qh. \label{rt-prop-6}
	\end{align}
Using \eqref{rt-prop-4}, \eqref{RT-dof-1}-\eqref{RT-dof-2} and
\eqref{rt-prop-6}, it is straightforward to show that
$\Pi^k_{RT}\q\cdot\n$ is continuous across element facets, so 
$\Pi^k_{RT}\q \in H(\dvr;\O)$.
Similarly, one can see that $\Pi^k_{RT}\q\cdot\n = 0$ on $\Gn$ if
$\q\cdot\n = 0$ on $\Gn$, so $\Pi^k_{RT}\q \in \V^k_{RT,h}$.
Details of these arguments can be found in
\cite{brezzi1991mixed,wang1994mixed,wheeler2006multipoint,arnold2005quadrilateral,ingram2010multipoint}.
		
\subsection{Enhanced Raviart-Thomas finite elements.}
In this section we develop a new family of enhanced Raviart-Thomas
spaces, which is used in our method. We present the definitions of
shape functions and degrees of freedom and discuss their
unisolvency. The idea of the construction is to enhance the
Raviart-Thomas spaces with bubbles that are curls of specially chosen
polynomials, so that each component of the velocity vector is of
dimension $Q^k(\R^d)$ and the velocity degrees of freedom can be
associated with the points of a tensor-product Gauss-Lobatto
quadrature rule. 

\subsubsection{Shape functions.}
In this subsection we adopt a convention for compact notation that 
$w^{-1} = 0$ for a polynomial variable $w$ unless it is multiplied by $w$. 
For example, it holds that 
$\xh^{-1} (\xh, \yh, \xh^2 \zh)^T = (1, 0, \xh \zh)^T$.
For $k \ge 1$ and integers $d_1$, $d_2$, $d_3$, define 
\begin{align*}
	\Bb^k_1(\Eh) = \mbox{span} \left\{ \begin{pmatrix}  \xh^{d_1}\yh^{d_2}\zh^{d_3} \\ 0 \\ 0 \end{pmatrix}:  0\le d_1,d_2,d_3 \le k, d_2 = k \mbox{ or } d_3 = k \right\}, \\
	\Bb^k_2(\Eh) = \mbox{span} \left\{ \begin{pmatrix}  0 \\ \xh^{d_1}\yh^{d_2}\zh^{d_3} \\ 0 \end{pmatrix}: 0\le d_1,d_2,d_3 \le k, d_1 = k \mbox{ or } d_3 = k \right\}, \\
	\Bb^k_3(\Eh) = \mbox{span} \left\{ \begin{pmatrix}  0 \\ 0 \\ \xh^{d_1}\yh^{d_2}\zh^{d_3} \end{pmatrix}: 0\le d_1,d_2,d_3 \le k, d_1 = k \mbox{ or } d_2 = k \right\}
\end{align*}
on the reference element $\Eh$. 
While the above construction was done explicitly in 3d, it
translates naturally to 2d by omitting the $\zh$ terms. 
We now define the space $\Bb^k$ as
\begin{align}\label{defn-B}
	\Bb^k(\Eh) =\bigcup_{i=1}^d \Bb^k_i .
	\end{align}
It is clear from the above definition that $\Q^k(\Eh,\R^d) =
\Vh^{k-1}_{RT}(\Eh) \oplus \Bb^k(\Eh)$ in both 2d and 3d.
	
For $\qh\in \Bb^k(\Eh)$, we then consider $\nablah \times (\xbh \times
\qh)$. Here, we use the regular curl and cross product operators in 3d. 
The cross product applies to a 2d vector by
representing the vector as a 3d one, with zeroed out third component,
resulting in a scalar function, i.e., 
$\xbh \times \qh = \xh q_2 - \yh q_1$ for $\qh = (q_1,q_2)^T$.
In 2d, $\nablah \times$ applies to a scalar function $\phi$ 
by representing the scalar function as a 3d
vector with zero first and second components, and the first and second
components of the result is defined as $\nablah \times \phi$, i.e.,
$\nablah \times \phi = (- \partial_2\phi, \partial_1\phi)^T$.
Therefore, if $\qh = (q_1,0)^T$ with $q_1 = \xh^{a_1} \yh^{a_2}$, 
$$ 
\nablah\times(\xbh\times\qh) = 
\xh^{a_1-1}\yh^{a_2}\begin{pmatrix} (a_2+1)\xh \\ -a_1\yh \end{pmatrix} . $$
We are now ready 
to construct a space isomorphic to $\Bb^k(\Eh)$, which is
better suited for the analysis as well as for practical implementation. 
More precisely, we define 
{\blue
\begin{align*}
	\Bbt^k_i(\Eh) = \nablah \times (\xbh \times \Bb^k_i(\Eh)) , 
\quad i = 1, \dots, d , \quad \mbox{ and } \quad \Bbt^k (\Eh) = \cup_{i=1}^d \Bbt^k_i(\Eh) . 
\end{align*}
}
One can check that in 2d,
	\begin{align}
	\Bbt^k_1(\Eh) &= \mbox{span}\left\{\xh^{a_1-1}\yh^{a_2}\begin{pmatrix} (a_2+1)\xh \\ -a_1\yh \end{pmatrix}: 0\le a_1,a_2 \le k, a_2=k \right\},\label{bt-2d-1} \\
	\Bbt^k_2(\Eh) &= \mbox{span}\left\{\xh^{b_1}\yh^{b_2-1}\begin{pmatrix} -b_2\xh \\ (b_1+1)\yh \end{pmatrix}: 0\le b_1,b_2 \le k, b_1=k \right\}\label{bt-2d-2},
	\end{align}
and in 3d,
	\begin{align}
	\Bbt^k_1(\Eh) &= \mbox{span}\left\{\xh^{a_1-1}\yh^{a_2}\zh^{a_3}\begin{pmatrix} (a_2+a_3+2)\xh \\ -a_1\yh \\ -a_1\zh \end{pmatrix}:
0\le a_1,a_2,a_3 \le k,  a_2=k \mbox{ or } a_3=k \right\}, \label{bt-3d-1}\\
	\Bbt^k_2(\Eh) &= \mbox{span}\left\{\xh^{b_1}\yh^{b_2-1}\zh^{b_3}\begin{pmatrix} -b_2\xh \\ (b_1+b_3+2)\yh \\ -b_2\zh \end{pmatrix}:
0\le b_1,b_2,b_3 \le k,  b_1=k \mbox{ or } b_3=k \right\}, \label{bt-3d-2}\\
	\Bbt^k_3(\Eh) &= \mbox{span}\left\{\xh^{c_1}\yh^{c_2}\zh^{c_3-1}\begin{pmatrix} -c_3\xh \\ -c_3\yh \\ (c_1+c_2+2)\zh \end{pmatrix}:
0\le c_1,c_2,c_3 \le k,  c_1=k \mbox{ or } c_2=k \right\}. \label{bt-3d-3}
	\end{align}
We define the enhanced Raviart-Thomas space $\Vh^k(\Eh) $ as 
	\begin{align} 
	\Vh^k(\Eh) = \Vh_{RT}^{k-1}(\Eh) + \Bbt^k(\Eh).  \label{Vh}
	\end{align}
	\begin{theorem}
	The sum \eqref{Vh} is a direct sum, i.e., $\Vh^k(\Eh) = \Vh_{RT}^{k-1}(\Eh) \oplus \Bbt^k(\Eh)$, and 
	$\dim \Vh^{k}(\Eh) = \dim\Qc^k(\Eh,\R^d)$.
	\end{theorem}
	\begin{proof}
We will prove that the space $\Bbt^k(\Eh)$ is isomorphic to $\Bb^k(\Eh)$. 
It suffices to show that the map $\qh \mapsto \nablah\times(\xbh\times\qh)$ is
	injective on $\Bb^k(\Eh)$. To see it, suppose that a linear
	combination of the elements of \eqref{bt-3d-1}-\eqref{bt-3d-3} is
	zero. Note that all elements in each space of
	\eqref{bt-3d-1}-\eqref{bt-3d-3} have distinct polynomials
	degrees. Therefore, for a component of fixed degrees of $\xh$, $\yh$,
	$\zh$ in the linear combination, only one element of each space is
	used to generate the component. This implies that
	
	\begin{align*}
	\alpha \xh^{a_1-1}\yh^{a_2}\zh^{a_3}\begin{pmatrix} (a_2+a_3+2)\xh \\ -a_1\yh \\ -a_1\zh \end{pmatrix} + \beta \xh^{b_1}\yh^{b_2-1}\zh^{b_3}\begin{pmatrix} -b_2\xh \\ (b_1+b_3+2)\yh \\ -b_2\zh \end{pmatrix} + \gamma \xh^{c_1}\yh^{c_2}\zh^{c_3-1}\begin{pmatrix} -c_3\xh \\ -c_3\yh \\ (c_1+c_2+2)\zh \end{pmatrix} = 0,
	\end{align*}
	with some coefficients $\alpha,\beta,\gamma$ and 
	\begin{align}
	a_1 = b_1 + 1 = c_1 + 1, \quad b_2 = a_2+1 = c_2 + 1, \quad c_3 = a_3+1 = b_3+1. \label{powers}
	\end{align}
	We will prove that $\alpha=\beta=\gamma=0$. If $a_2 = k$, then
	$\beta=0$ due to $0\le a_i,b_i,c_i \le k$ and
	\eqref{powers}. Comparing the components of the above equation, we have
	\begin{align*}
	-\alpha a_1 - \gamma(a_3+1) = 0, \quad -\alpha a_1 + \gamma(a_1 + a_2 + 1) = 0,
	\end{align*}
	and therefore $\alpha=\gamma = 0$. Similarly, $\gamma = 0$ if $a_3 = k$ due to \eqref{powers}, and a similar argument gives
	\begin{align*}
	-\alpha a_1 - \beta(a_3+1) = 0, \quad -\alpha a_1 + \beta(a_1 + a_2 + 1) = 0,
	\end{align*}
which results in $\alpha = \beta = 0$. Since this argument holds for
any component of the same polynomial degrees, the map $\qh \mapsto
\nablah\times(\xbh\times\qh)$ is injective on $\Bb^k(\Eh)$, and therefore
it is an isomorphism from $\Bb^k(\Eh)$ to $\Bbt^k(\Eh)$. This implies that
the monomials in \eqref{bt-2d-1}--\eqref{bt-2d-2} and
\eqref{bt-3d-1}--\eqref{bt-3d-3} form a basis of $\Bbt^k(\Eh)$.  Note
that every element of $\Bbt_i^k(\Eh)$ in
\eqref{bt-2d-1}--\eqref{bt-3d-3} contains at least one entry with a
variable of degree $k+1$, therefore 
$\Vh^{k-1}_{RT}(\Eh) \cap \Bbt^k(\Eh) = \{0\}$, i.e., the sum
\eqref{Vh} is a direct sum. This implies that $\dim \Vh^{k}(\Eh) =
\dim\Qc^k(\Eh,\R^d)$.
\end{proof}

	\subsubsection{Degrees of freedoms and unisolvency.}
Using the definition \eqref{Vh} of $\Vh^k(\Eh)$ and the definitions of 
$\Vh^{k-1}_{RT}(\Eh)$ and $\Bbt^k(\Eh)$, we have that for $\qh \in \Vh^k(\Eh)$,
	\begin{align*}
	&\mbox{ in 2d: } q_1 \in \Pc^{k+1}(\xh)\otimes \Pc^{k}(\yh), \quad q_2 \in \Pc^{k+1}(\yh)\otimes \Pc^{k}(\xh), \\
	&\mbox{ in 3d: } q_1 \in \Pc^{k+1}(\xh)\otimes \Qc^{k}(\yh,\zh), \quad q_2 \in \Pc^{k+1}(\yh)\otimes \Qc^{k}(\xh,\zh), \quad q_3 \in \Pc^{k+1}(\zh)\otimes \Qc^{k}(\xh,\yh).
	\end{align*}
	For the degrees of freedom of $\Vh^k$ we consider the following moments:
	\begin{align}
&\mbox{for } k \ge 1, \quad \qh \mapsto \int_{\eh} \qh\cdot\nh_{\eh} \,\hat{p}, 
\quad \forall \hat{p}\in \Qc^k(\eh), \forall \eh \in \partial \Eh, \label{dof1}\\
&\mbox{for } k \ge 2, \quad \qh \mapsto \int_{\Eh} \qh \cdot \ph , \quad \forall \ph \in 
	\begin{cases}	
	\begin{pmatrix} \Pc^{k-2}(\xh)\otimes\Pc^{k}(\yh) \\ \Pc^{k-2}(\yh)\otimes\Pc^{k}(\xh) \end{pmatrix} \quad \mbox{in 2d},\\
	\begin{pmatrix} \Pc^{k-2}(\xh)\otimes\Qc^{k}(\yh,\zh) \\ \Pc^{k-2}(\yh)\otimes\Qc^{k}(\xh,\zh) \\
	\Pc^{k-2}(\zh)\otimes\Qc^{k}(\xh,\yh) \end{pmatrix} \quad \mbox{in 3d}.
	\end{cases} \label{dof2}
	\end{align}
	
	The number of degrees of freedom given by \eqref{dof1} and
	\eqref{dof2} are $2d(k+1)^{d-1}$ and $d(k-1)(k+1)^{d-1}$, respectively. Therefore the total number of
	DOFs is $d(k+1)^d$, which is same as the $\mbox{dim}\Qc^k(\Eh,\R^d)$. 
	We notice, that similarly to classical mixed
	finite elements such as the Raviart-Thomas or Brezzi-Douglas-Marini
	families of elements, the first set of moments \eqref{dof1} stands for
	facet DOFs, which will be required to be
	continuous across the facet. The second set of moments
	\eqref{dof2} represents interior DOFs, and no
	continuity requirements will be imposed on these. These
	new elements can be viewed as the Raviart-Thomas family with added
	bubbles, which are curls of specially chosen polynomials.
	\begin{theorem}
	Let $\Vh^k(\Eh)$ be defined as in \eqref{Vh}.
	For $\vh \in \Vh^k(\Eh)$ suppose that the evaluations of DOFs
	\eqref{dof1} and \eqref{dof2} are all zeros. Then $\vh = 0$.
	\end{theorem} 
	\begin{proof}
Without loss of generality, we present the proof for $\Eh = [-1,1]^d$.
We prove the theorem in 3d, while the 2d result can be obtained in the
same manner. 
From the definition of shape functions of $\Vh^k(\Eh)$,
$\vh\cdot \nh_{\eh}\in \Qc^k(\eh)$ for a face $\eh$ of
$\Eh$. Therefore, vanishing DOFs \eqref{dof1} imply that
	\begin{align}
	\vh = \begin{pmatrix} v_1\\v_2\\v_3 \end{pmatrix} = \begin{pmatrix} (1-\xh^2)\tilde{v}_1(\xh,\yh,\zh) \\ (1-\yh^2)\tilde{v}_2(\xh,\yh,\zh) \\ (1-\zh^2)\tilde{v}_3(\xh,\yh,\zh) \end{pmatrix}, \label{dof-red1}
	\end{align}
	with 
	$$ \tilde{v}_1 \in \Pc^{k-1}(\xh)\otimes \Qc^{k}(\yh,\zh), \quad \tilde{v}_2 \in \Pc^{k-1}(\yh)\otimes \Qc^{k}(\xh,\zh), \quad \tilde{v}_3 \in \Pc^{k-1}(\zh)\otimes \Qc^{k}(\xh,\yh). $$
	In addition, the vanishing DOFs \eqref{dof2} further reduce 
$\tilde{v}_i,\,i=1,2,3,$ to 
	\begin{align}
	\tilde{v}_1 = L_w^{k-1}(\xh)w_1(\yh,\zh), \qquad \tilde{v}_2 = L_w^{k-1}(\yh)w_2(\xh,\zh), \qquad \tilde{v}_3 = L_w^{k-1}(\zh)w_3(\xh,\yh), 
\label{dof-red2}
	\end{align}
where $w_1 \in \Qc^{k}(\yh,\zh)$, etc., and $L_w^{k-1}(t)$ is the monic
polynomial of degree $k-1$ on $[-1,1]$ orthogonal to
$\Pc^{k-2}(t)$ with weight $(1-t^2)$.  Since all monomials in
$\Vh^k(\Eh)$ are of degree $\le 3k$, $\yh^k\zh^k$ is not contained in
$w_1(\yh,\zh)$. Similar statements hold with $\zh^k \xh^k$, $\xh^k
\yh^k$ and $w_2(\xh, \zh)$, $w_3(\xh, \yh)$, respectively. Therefore
we can write
	\begin{align*}
	w_1(\yh, \zh) = \yh^k p_1(\zh) + \zh^k q_1(\yh) + \tilde{w}_1(\yh, \zh) , \qquad p_1 \in \Pc^{k-1} (\zh), q_1 \in \Pc^{k-1} (\yh),  \tilde{w}_1(\yh, \zh) \in \Qc^{k-1}(\yh, \zh),
	\end{align*}
and similar expressions are available for $w_2$ and $w_3$.  If $p_1
\not = 0$, $v_1$ has monomials with factor $\xh^{k+1} \yh^k$. From the
forms of $\Bbt^k_i(\Eh)$, $i=1,2,3$, this can be obtained only from a
linear combination of elements in $\Bbt^k_3(\Eh)$ with $c_1 = c_2 = k$.
However, a linear combination of elements in $\Bbt^k_3(\Eh)$ which
gives $\xh^{k+1}\yh^k p_1(\zh)$ in the first component also has the
third component $-(2k+2) \xh^k \yh^k P_1(\zh)$ where $P_1(\zh)$ is the
anti-derivative of $p_1(\zh)$ with $P_1(0) = 0$. All terms in $v_3$
having $\xh^k \yh^k$ as a factor are obtained only from
$\Bbt^k_3(\Eh)$. Furthermore, $v_3$ does not contain any terms with
factor $\xh^k \yh^k$ due to the form of $w_3$ we discussed, therefore
$P_1 = 0$ and $p_1 = 0$ as well. Applying a similar argument we can
conclude that $q_1 = 0$, so $w_1 \in \Qc^{k-1}(\yh, \zh)$. In addition, we
can show that $w_2 \in \Qc^{k-1}(\xh, \zh)$ and $w_3 \in
\Qc^{k-1}(\xh, \yh)$ by similar arguments. 

We now claim that
$\dvrg \vh = 0$. First, $\dvrg \vh \in \Qc^{k-1}(\Eh)$ holds from the
definition of the shape functions. Then the Green's identity and the
vanishing DOFs assumption give
\begin{align}\label{green}
\int_{\Eh} \dvrg \vh q \,d\xbh 
= \int_{\partial \Eh} \vh \cdot \n\, q \,d \hat s - \int_{\Eh} \vh \cdot \nabla q \,d\xbh = 0
\end{align}
	for any $q \in \Qc^{k-1}(\Eh)$. In particular $q = \dvrg \vh$ gives $\dvrg \vh = 0$. 
From the expression of $\vh$ in \eqref{dof-red2},
	\begin{align*}
	0 = \dvrg \vh = \tilde{L}^k(\xh) w_1(\yh, \zh) + \tilde{L}^k (\yh) w_2(\xh, \zh) + \tilde{L}^k (\zh) w_3(\xh, \yh) 
	\end{align*}
	where $\tilde{L}^k (t) = \frac{d}{dt} ((1-t^2) L_w^{k-1}(t))$. For $0 \le i \le k-1$, note that 
	\begin{align*}
	\int_{-1}^1 \tilde{L}^k(t) t^i \, dt = - i \int_{-1}^1 (1-t^2) L_w^{k-1}(t) t^{i-1} \,dt = 0 
	\end{align*}
by integration by parts and the definition of $L_w^{k-1}$. From this observation we can obtain
	\begin{align*}
0 = \int_{\Eh} (\dvrg \vh) \tilde{L}^k (\xh) w_1(\yh, \zh) \,d\xbh 
= \int_{\Eh} (\tilde{L}^k (\xh) w_1(\yh, \zh))^2 \,d\xbh,
	\end{align*}
	which implies $w_1 = 0$. We can conclude $w_2 = w_3 = 0$ with similar arguments, therefore $\vh = 0$.
	\end{proof}

\subsubsection{Mixed finite element spaces.}
For $k\ge 1$, consider the pair of mixed finite element spaces 
$\Vh^k(\Eh) \times \Wh^{k-1}(\Eh)$, recalling that 
\begin{align*}
	\Vh^k(\Eh) = \Vh^{k-1}_{RT}(\Eh)\oplus\Bbt^k(\Eh), \quad
	\Wh^{k-1}(\Eh) = \Qc^{k-1}(\Eh).
\end{align*}
Note that the construction of $\Vh^k(\Eh)$ and \eqref{rt-prop-0} imply that
\begin{align}\label{rt-prop-00}
\nablah\cdot\Vh^k(\Eh) = \Wh^{k-1}(\Eh),\quad \mbox{and} \quad
\forall\vh\in\Vh^k(\Eh),\, \forall\eh\subset\partial\Eh,\,
\vh\cdot\nh_{\eh}\in \Qc^k(\eh).
\end{align}
Recall also that $\mbox{dim}\Vh^k(\Eh) = \mbox{dim}\Qc^k(\Eh,\R^d) = d(k+1)^d$ 
and that its degrees of freedom are the moments \eqref{dof1} and \eqref{dof2}.
We consider an alternative definition of degrees of freedom involving the 
values of vector components at the Gauss-Lobatto quadrature
points; see Figure \ref{fig:2_1}, where filled arrows indicate the
facet degrees of freedom for which continuity across facets
is required, and unfilled arrows represent the "interior" degrees of
freedom, local to each element. We have omitted some of the degrees of
freedom from the backplane of the cube for clarity of
visualization. This choice gives certain orthogonalities for the
Gauss-Lobatto quadrature rule which we will discuss in details in the
forthcoming subsections.

 		\begin{figure}[ht!]
 			\centering
 			\begin{subfigure}[b]{0.25\textwidth}
 				\includegraphics[width=\textwidth]{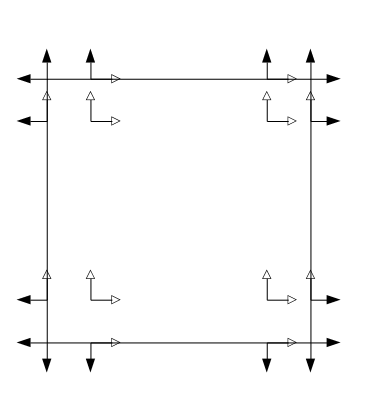}
 				\caption{$\Vh^3(\Eh)$ in 2d}\label{fig:2_1-1}
 			\end{subfigure}
 			\begin{subfigure}[b]{0.3\textwidth}
 				\includegraphics[width=\textwidth]{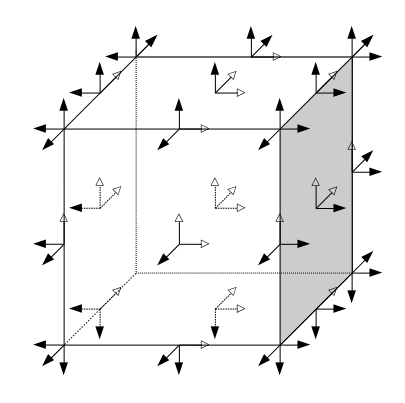}
 				\caption{$\Vh^2(\Eh)$ in 3d}\label{fig:2_1-2}
 			\end{subfigure}
 			\caption{Degrees of freedom of the enhanced Raviart-Thomas elements}\label{fig:2_1}
 		\end{figure}
		

The unisolvency of the enhanced Raviart-Thomas spaces shown in the
previous section implies the existence of a unique projection operator
$\Pih_*^k: H^1(\Eh,\R^d)\to\Vh^k(\Eh)$ such that
\begin{align}
&\mbox{for } k \ge 1, \quad \gnp[(\Pih_*^k\qh - \qh)\cdot\n_{\eh}]{\hat{p}}_{\eh} = 0, \quad
\forall \eh\subset \partial \Eh, \, 
\forall \hat{p}_k\in \Qc^k(\eh),\label{rte-prop-1} \\
&\mbox{for } k \ge 2, \quad \inp[\Pih_*^k\qh - \qh]{\ph }_{\Eh} = 0, \quad \forall\ph \in 
	\begin{cases}
	\begin{pmatrix} \Pc^{k-2}(\xh)\otimes\Pc^{k}(\yh) \\ \Pc^{k-2}(\yh)\otimes\Pc^{k}(\xh) \end{pmatrix} \quad \mbox{in 2d},\\
	\begin{pmatrix} \Pc^{k-2}(\xh)\otimes\Qc^{k}(\yh,\zh) \\ \Pc^{k-2}(\yh)\otimes\Qc^{k}(\xh,\zh) \\
	\Pc^{k-2}(\zh)\otimes\Qc^{k}(\xh,\yh) \end{pmatrix} \quad \mbox{in 3d}.
	\end{cases}\label{rte-prop-2}
	\end{align}
The Green's identity \eqref{green} together with \eqref{rte-prop-1} and 
\eqref{rte-prop-2} implies that 
\begin{align}\label{rte-prop-div}
	\inp[\nablah\cdot(\Pih_*^k\qh-\qh)]{\wh}_{\Eh} = 0, \quad \forall \wh\in\Wh^{k-1}(\Eh).
\end{align}
Using \eqref{rt-prop-3}, the above implies that
\begin{align}\label{rte-prop-div-E}
\inp[\nabla\cdot(\Pi_*^k\q-\q)]{w}_E = 0, \quad \forall w\in W^{k-1}(E).
\end{align}

Let $\V_h^k\times W_h^{k-1}$ be the pair of enhanced Raviart-Thomas
spaces on $\Tc_h$ defined as in \eqref{rt-prop-5} and the
projection operator $\Pi_*^k$ from $\V\cap H^1(\O,\R^d)$ onto $\V^k_h$
be defined via the Piola transformation as in \eqref{rt-prop-6}.
\begin{lemma}\label{lemma-inf-sup}
There exists a positive constant $\beta$, independent of $h$, such that
		\begin{align}
		\inf\limits_{w\in W_h^{k-1}}\sup\limits_{\q\in \V^k_h} \frac{\inp[\dvrg \q]{w}}{\|w\|\|\q\|_{\dvr}} \geq \beta . \label{inf-sup}
		\end{align}
	\end{lemma}
	\begin{proof}
We consider the auxiliary problem
\begin{align}
\dvrg\boldsymbol{\psi} = w\quad \mbox{in } \O, \quad
\boldsymbol{\psi} = \mathbf{g} \quad \mbox{on } \partial\O, \label{aux-1}
\end{align}
where $\mathbf{g}\in H^{1/2}(\partial\O,\R^d)$ is constructed such
that it satisfies $\int_{\partial\O}\mathbf{g}\cdot\n = \int_{\O}w$
and $\mathbf{g}\cdot \n = 0$ on $\Gn$. More specifically, we choose
$\mathbf{g} = (\int_{\partial\O}w)\phi\n$, where $\phi \in
C^0(\partial\O)$ is such that $\int_{\partial\O}\phi = 1$ and $\phi =
0$ on $\Gn$. Clearly, such construction implies
$\|\mathbf{g}\|_{1/2,\partial\O} \le C\|w\|$. It is known
\cite{galdi2011introduction} that the problem
\eqref{aux-1} has a solution satisfying
\begin{align}
\|\boldsymbol{\psi}\|_1 \le C\left( \|w\| 
+ \|\mathbf{g}\|_{1/2,\partial\O}\right) \le C\|w\|.
\end{align}
As the solution $\boldsymbol{\psi}$ is regular enough,
$\Pi_*^k\boldsymbol{\psi}$ is well defined. Using \eqref{rte-prop-div-E}, 
the choice $\q = \Pi_*^k\boldsymbol{\psi} \in \V^k_h$ yields
		\begin{align*}
		\inp[\dvrg \q]{w} = \inp[\dvrg\Pi_*^k\boldsymbol{\psi}]{w} = \inp[\dvrg\boldsymbol{\psi}]{w} = \|w\|^2.
		\end{align*}
We complete the proof by {\blue exploiting} the continuity bound
$ \|\Pi_*^k \boldsymbol{\psi}\|_{\dvr} \leq C\|\boldsymbol{\psi}\|_1$, 
which is stated in \eqref{div-bound} below.
\end{proof}
	We also note that since $\V_{RT,h}^{k-1}\subset \V_h^k$, it 
follows from the definition of $\Pi_{RT}^k$ that
\begin{align}
\nabla \cdot \v = \nabla \cdot \Pi^{k-1}_{RT}\v, 
\quad \forall \v \in \V^k_h \label{div-preserv},\\
\|\Pi^{k-1}_{RT}\v\| \leq C \|\v\|, \quad \forall \v \in \V^k_h \label{rt-cont}.
\end{align}
	
	
\subsection{Quadrature rule.}
We next present the quadrature rule for the velocity bilinear
form, which is designed to allow for local velocity elimination around
finite element nodes. We perform the integration on any element by
mapping to the reference element $\Eh$. The quadrature rule is defined
on $\Eh$. We have for $\q, \,\v\in\V^k_h$,
\begin{align*}
\int_E \Ki\q\cdot\v\,d\x &= \int_{\Eh} \hat{K}^{-1} 
\frac{1}{J_E}DF_E\qh\cdot\frac{1}{J_E}DF_E\vh\, J_Ed\xbh 
= \int_{\Eh}\frac{1}{J_E}DF_E^T\hat{K}^{-1}DF_E \qh\cdot\vh\,d\xbh 
\equiv \int_{\Eh}\Kci\qh\cdot\vh\,d\xbh,
\end{align*}
	where
	\begin{align}
	\mathcal{K} = J_EDF_E^{-1}\hat{K}(DF^{-1}_E)^{T}. \label{def-kci}
	\end{align}
It is straightforward to show that \eqref{perm-bounds} and \eqref{mapping-bounds} 
imply that
	\begin{align}
	\|\Kc\|_{0,\infty,\Eh} \sim h^{d-2}\|K\|_{0,\infty,E},\quad \|\Kci\|_{0,\infty,\Eh} \sim h^{2-d}\|\Ki\|_{0,\infty,E}. \label{perm-mapped-bounds}
	\end{align}
Let $\Xi_k := \{\xi_k(i)\}_{i=0}^{k}$ and $\Lambda_k :=
\{\lambda_k(i)\}_{i=0}^k$ be the points and weights of the
Gauss-Lobatto quadrature rule on $[-1,1]$.  If $k$ is clear in
context, we use $(p, q)_Q$ to denote the evaluation of Gauss-Lobatto
quadrature with $k+1$ points for $(p,q)$.  We also define
	\begin{align}
	\ph_{\bs{i}} := (\xi_k(\bs{i}_1), ..., \xi_k(\bs{i}_d)), 
\quad w_k(\bs{i}) := \lambda_k(\bs{i}_1) \cdots \lambda_k(\bs{i}_d) 
\qquad \text{for }\quad \bs{i} \in \Ic_k \equiv
\{ (\bs{i}_1, ..., \bs{i}_d), \,\, \bs{i}_j \in \{0,...,k\} \}. \label{r-lambda-def}
	\end{align}
	For the method of order $k$, the quadrature rule is defined on an element $E$ as follows
	\begin{align}
	\inp[\Ki\q]{\v}_{Q,E} \equiv \inp[\Kci\qh]{\vh}_{\hat{Q},\Eh} 
\equiv \sum_{\bs{i} \in \Ic_k} w_k({\bs{i}})\Kci(\ph_{\bs{i}})
\qh(\ph_{\bs{i}})\cdot\vh(\ph_{\bs{i}}) . \label{quad-rule-def}
	\end{align}
	The global quadrature rule can then be defined as
	\begin{align*}
	\inp[\Ki\q]{\v}_Q \equiv \sum_{E\in\Tc_h}\inp[\Ki\q]{\v}_{Q,E}.
	\end{align*}
Note that the method in the lowest order case $k=1$ is very similar in
nature to the one developed in
\cite{wheeler2006multipoint,ingram2010multipoint}, although we use
different finite element spaces.

We next show that the evaluation at the tensor-product quadrature points is a set of DOFs of
$\Vh^k(\Eh)$, so the bilinear form with the quadrature is not degenerate.
\begin{lemma} \label{quad-vanish}
For $p \in \Qc^k(\Eh)$, if the evaluations of $p$ vanish at all the
quadrature nodes of the tensor product Gauss--Lobatto rules on $\Eh$,
then $p=0$.
\end{lemma}
	
The above statement is obvious, because the evaluations at the tensor product quadrature nodes 
become a set of DOFs of $\Qc^k(\Eh)$.
\begin{lemma}
For $\qh\in\Vh^k(\Eh)$, if $\qh(\ph_{\bs{i}}) = 0$ 
for all $\ph_{\bs{i}}$ in \eqref{r-lambda-def}, then $\qh = 0$. 
\end{lemma}
\begin{proof}
Without loss of generality, we present the proof for $\Eh = [-1,1]^d$.
It suffices to show that the vanishing quadrature evaluation
assumption implies that the moments in \eqref{dof1} and
\eqref{dof2} vanish. Since $\qh \cdot \n_e \in \Qc^k(e) \,\, \forall \, e \subset
\partial \Eh$, the vanishing quadrature assumption for nodes on $e$ implies 
that $\qh \cdot \n_e = 0$. Therefore the moments in \eqref{dof1} vanish and 
$\qh$ is reduced to the form in \eqref{dof-red1}, i.e.,
	\begin{align*}
	\qh = \begin{pmatrix} q_1\\q_2\\q_3 \end{pmatrix} = \begin{pmatrix} (1-\xh^2)\tilde{q}_1(\xh,\yh,\zh) \\ (1-\yh^2)\tilde{q}_2(\xh,\yh,\zh) \\ (1-\zh^2)\tilde{q}_3(\xh,\yh,\zh) \end{pmatrix}, 
	\end{align*}
	with
	$$ \tilde{q}_1 \in \Pc^{k-1}(\xh)\otimes \Qc^{k}(\yh,\zh), \quad \tilde{q}_2 \in \Pc^{k-1}(\yh)\otimes \Qc^{k}(\xh,\zh), \quad \tilde{q}_3 \in \Pc^{k-1}(\zh)\otimes \Qc^{k}(\xh,\yh). $$
We want to show that all moments \eqref{dof2} of $\qh$ are zeros.
To do it, we first express $\tilde{q}_1$ as
	\begin{align}
	\tilde{q}_1 = \sum_{j=0}^{k-1} L_w^{j}(\xh) r_j(\yh, \zh), \qquad  r_j(\yh, \zh) \in \Qc^k(\yh,\zh), \label{q1-expan}
	\end{align}
where $L_w^j$ is the Legendre polynomial of degree $j$ with weight
$(1-\xh^2)$ as before. For fixed $\yh$ and $\zh$, let us consider the
Gauss-Lobatto quadrature of $q_1 v$ along $\xh$ with $v \in
\Pc^{k-2}(\xh)$. For fixed values of $\yh$ and $\zh$, $q_1$ is a
polynomial of degree $\le k+1$, so this quadrature evaluation of $q_1
v$ equals the integration of $q_1 v$ in $\xh$ with the fixed $\yh$ and
$\zh$. In particular, if $v = L_w^m(\xh)$, $0 \le m \le k-2$, $\yh =
\xi_k(i)$, $\zh = \xi_k(j)$, then the vanishing quadrature assumption
and the expression of $\tilde{q}_1$ in \eqref{q1-expan} give
	\begin{align*}
	0 = \sum_{l=0}^k \lambda_k(l) q_1 (\xi_k(l), \xi_k(i), \xi_k(j)) v(\xi_k(l)) = \int_{-1}^1 q_1 (\xh , \xi_k(i), \xi_k(j)) v(\xh))\, d\xh = \int_{-1}^1 (1-\xh^2) (L_w^m(\xh))^2 r_m(\xi_k(i), \xi_k(j)).
	\end{align*}
	This implies that $r_m(\yh, \zh ) = 0$ for any $\yh = \xi_k(i)$, $\zh = \xi_k(j)$, $0 \le i,j \le k$ if $0 \le m \le k-2$, and therefore $r_m = 0$ for $0 \le m \le k-2$ by Lemma~\ref{quad-vanish}. As a consequence, $q_1 = (1-\xh^2) L_w^{k-1}(\xh) r_{k-1}(\yh, \zh)$ with $r_{k-1} \in \Qc^k(\yh, \zh)$ and its evaluations at the DOFs given by the first component in \eqref{dof2} vanish. We can derive similar results for $q_2$ and $q_3$, i.e., $\qh$ gives only vanishing moments for the DOFs \eqref{dof2}. We can conclude that $\qh = 0$ by the same argument as in the previous proof of unisolvency.
	\end{proof}

The above result allows us to define a set of DOFs of $\Vh^k(\Eh)$ as
the evaluations of the vectors at the tensor-product quadrature points
$\ph_{\bs{i}}$, $\bs{i} \in \Ic_k$. Examples were given in
Figure~\ref{fig:2_1}. Recall that for points on $\partial\Eh$, some of
the vector components are facet degrees of freedom for which
continuity across facets is required, while some are "interior"
degrees of freedom, local to each element. For convenience of
notation, denote the set of points $\ph_{\bs{i}}$ by $\ph_i$,
$i=1,\dots,n_k$, $n_k = (k+1)^d$. Any vector $\qh(\ph_i)$ at the node
$\ph_i$ is uniquely determined by its $d$ components evaluated at
this node. Since we chose the Gauss-Lobatto (or trapezoid, when $k=1$)
quadrature points for the construction of the velocity degrees of freedom, 
we are guaranteed to have $d$
orthogonal DOFs associated with each node (quadrature point) $\ph_i$,
and they uniquely determine the nodal vector $\qh(\ph_i)$. More
precisely,
\begin{align}
\qh(\ph_i) = \sum_{j=1}^{d} (\qh\cdot\nh_{ij})(\ph_i)\nh_{ij},
\end{align}
where $\nh_{ij},\,j=1,\dots,d$, are the outward unit normal vectors to
the $d$ hyperplanes of dimension $(d-1)$ that intersect at $\ph_i$,
each one parallel to one of the three mutually orthogonal
facets of the reference element. Denote the velocity basis functions
associated with $\ph_i$ by $\vh_{ij},\,j=1,\dots,d$, i.e.,
\begin{align}\label{basis}
	(\vh_{ij}\cdot\nh_{ij})(\ph_i) = 1,\quad (\vh_{ij}\cdot\nh_{im})(\ph_i) = 0,\, m\neq j,\:\mbox{and}\: (\vh_{ij}\cdot\nh_{lm})(\ph_l) = 0, l\neq i,\, m=1,\dots,d.
\end{align}
The quadrature rule \eqref{quad-rule-def} couples only $d$ basis
functions associated with a node. For example, in 3d, for any node 
$i=1,\dots,n_k$,
	\begin{align}
	&\inp[\Kci\vh_{i1}]{\vh_{i1}}_{\hat{Q},\Eh} = \Kci_{11}(\ph_i)w_k(i) , 
\quad \inp[\Kci\vh_{i1}]{\vh_{i2}}_{\hat{Q},\Eh} = \Kci_{21}(\ph_i)w_k(i),
\nonumber \\
	&\inp[\Kci\vh_{i1}]{\vh_{i3}}_{\hat{Q},\Eh} = \Kci_{31}(\ph_i)w_k(i), 
\quad \inp[\Kci\vh_{i1}]{\vh_{mj}}_{\hat{Q},\Eh} = 0 \quad \forall mj 
\neq i1,i2,i3. \label{local}
	\end{align}
By mapping back \eqref{quad-rule-def} to the physical element $E$, we obtain
	\begin{align}
	\inp[\Ki\q]{\v}_{Q,E} = \sum_{i=1}^{n_k}J_E(\ph_i)w_k(i)\Ki(\p_i)\q(\p_i)\cdot\v(\p_i). \label{quad-def-phys}
	\end{align}
	
	Denote the element quadrature error by
	\begin{align}
	\s_E\inp[K^{-1}\q]{\v} \equiv \inp[\Ki\q]{\v}_E - \inp[\Ki\q]{\v}_{Q,E},
	\end{align}
	and define the global quadrature error by $\sigma\inp[\Ki\q]{\v}\big|_E = \s_E\inp[K^{-1}\q]{\v}$. Similarly, denote the quadrature error on the reference element by
	\begin{align}
	\hat{\s}_E\inp[\Kci\qh]{\vh} \equiv \inp[\Kci\qh]{\vh}_{\Eh} - \inp[\Kci\qh]{\vh}_{\hat{Q},\Eh}.
	\end{align}
	The following lemma will be used to bound the quadrature error.
	\begin{lemma}
		For any $\qh\in\Vh^k(\Eh)$ and for any $k \ge 1$,
		\begin{align}
		\inp[\qh - \Pih^{k-1}_{RT}\qh]{\vh}_{\hat{Q},\Eh} = 0, 
\quad\mbox{for all vectors } \vh\in\Qc^{k-1}(\Eh,\R^d).\label{exactness-quad}
		\end{align}
	\end{lemma} 
	\begin{proof}
Without loss of generality, we present the proof for $\Eh = [-1,1]^d$.
	We show a detailed proof only for the 3d case because the 2d case is similar. Let $v_i$, $i=1,2,3$ be the $i$-th component of $\qh - \Pih^{k-1}_{RT}\qh$.
	Considering the expression $v_1$ with the basis of Legendre polynomials, the definition of shape functions in $\Vh^k(\Eh)$ and the constraints from \eqref{RT-dof-2} yield that $v_1$ has the form
	\begin{align}
	v_1 = L^{k-1}(\xh) p_1(\yh, \zh) + L^k(\xh) q_1(\yh, \zh) + L^{k+1}(\xh) r_1(\yh, \zh) + L^k(\yh) u_1 (\xh, \zh) + L^k(\zh) w_1 (\xh, \yh) 
	\end{align}
	where $L^i$ is the standard $i$-th Legendre polynomial as before, $p_1, q_1, r_1 \in \Qc^{k-1}(\yh, \zh)$, 
	\begin{align}
	u_1 \in \Pc^{k+1}(\xh) \otimes \Pc^{k-1} (\zh) + \Qc^k(\xh, \zh) , \qquad w_1 \in \Pc^{k+1}(\xh) \otimes \Pc^{k-1} (\yh) + \Qc^k(\xh, \yh) . \label{u1-w1-express}
	\end{align}
	From \eqref{RT-dof-1}, the restrictions of $v_1$ on $\xh = -1$ and on $\xh = 1$ are orthogonal to $\Qc^{k-1}(\yh, \zh)$, and it gives two equations 
	\begin{align}
	p_1 + q_1 + r_1 = 0, \qquad p_1 - q_1 + r_1 = 0, 
	\end{align}
	therefore $q_1 = 0$ and $r_1 = -p_1$. A similar argument can be applied to $v_2$ and $v_3$. In summary, we have 
	\begin{align}
	v_1 &= (L^{k-1}(\xh) - L^{k+1}(\xh)) p_1(\yh, \zh) +  L^k(\yh) u_1 (\xh, \zh) + L^k(\zh) w_1 (\xh, \yh) , \\
	v_2 &= (L^{k-1}(\yh) - L^{k+1}(\yh)) p_2(\zh, \xh) +  L^k(\zh) u_2 (\xh, \yh) + L^k(\xh) w_2 (\yh, \zh) , \\
	v_3 &= (L^{k-1}(\zh) - L^{k+1}(\zh)) p_3(\xh, \yh) +  L^k(\xh) u_3 (\yh, \zh) + L^k(\yh) w_3 (\yh, \zh), 
	\end{align}
where $u_2$, $u_3$, $w_2$, $w_3$ belong to polynomial spaces similar to the spaces 
in \eqref{u1-w1-express} with variable permutation. 
To prove $(v_1, q)_{\hat{Q}, \Eh} = 0$ for $q \in \Qc^{k-1}(\Eh)$, we will show 
	\begin{align}
	((L^{k-1}(\xh) - L^{k+1}(\xh)) p_1(\yh, \zh), q)_{\hat{Q}, \Eh} = 0, \quad ( L^k(\yh) u_1 (\xh, \zh), q)_{\hat{Q}, \Eh} = 0, \quad (L^k(\zh) w_1 (\xh, \yh) , q)_{\hat{Q}, \Eh} = 0 . \label{Qhat-vanish}
	\end{align}
For the first equality, recall that the quadrature points of the
Gauss-Lobatto rules are the two endpoints and the zeros of
$\frac{d}{dt} L^k(t)$ in $[-1, 1]$. It is clear that $L^{k-1} -
L^{k+1}$ vanishes at the two endpoints. In addition, $L^{k-1} - L^{k+1}$
vanishes at the zeros of $\frac{d}{dt} L^k(t)$ in $[-1,1]$ from the
identities
	\begin{align*}
	(k+1) (L^{k+1} - L^{k-1})(t) = (2k+1) (t L^k(t) - L^{k-1}(t)) = (2k+1) \frac{t^2-1}{k} \frac{d}{dt} L^k(t) . 
	\end{align*}
Therefore, the first equality in \eqref{Qhat-vanish} holds. To prove
the second equality in \eqref{Qhat-vanish}, let us consider a
restriction of the tensor product Gauss-Lobatto rule for fixed
quadrature points of $\xh$ and $\zh$.  For fixed $\xh$ and $\zh$, the
product $L^k(\yh) u_1(\xh, \zh) q(\xh, \yh, \zh)$ is a polynomial in
$\yh$ of degree at most $2k-1$, so evaluation of $L^k(\yh) u_1(\xh,
\zh) q(\xh, \yh, \zh)$ with the restricted Gauss-Lobatto rule is the same
as the integration of the function in $\yh$. However, this integration
in $\yh$ is zero because $L^k(\yh)$ and $q \in \Qc^{k-1}(\xh, \yh,
\zh)$ are orthogonal. Since $(\cdot, \cdot)_{\hat{Q}, \Eh}$ is a sum
of these restricted Gauss-Lobatto rules, $( L^k(\yh) u_1 (\xh, \zh),
q)_{\hat{Q}, \Eh} = 0$. The third equality in \eqref{Qhat-vanish}
follows from the same argument as the second equality. Finally,
the same argument can be used for $v_2$ and $v_3$, so the assertion 
is proved.
\end{proof}

	\subsection{The $k$-th order MFMFE method.}

We first define an appropriate projection to be used in the method for
the Dirichlet boundary data $g$. This is necessary for optimal
approximation of the boundary condition term. Moreover, the numerical
tests suggest that this is not a purely theoretical artifact, as
without the projection we indeed see a deterioration in the rates of
convergence.  For a facet $\eh \in \partial \Eh$, let $\hat\Rc^{k-1}_{\eh}$ be
the $L^2(\eh)$-orthogonal projection onto $\Qc^{k-1}(\eh)$, satisfying for any
$\hat \phi \in L^2(\eh)$, 
$$
\gnp[\hat\phi - \hat\Rc^{k-1}_{\eh}\hat\phi]{\wh}_{\eh} = 0 \quad \forall \, \wh \in \Qc^{k-1}(\eh).
$$
Let $\Rc^{k-1}_h : L^2(\partial\O) \to W_h^{k-1}|_{\partial\O}$ be such 
that for any
$\phi \in L^2(\partial\O)$, $\Rc^{k-1}_h \phi = \hat\Rc^{k-1}_{\eh}\hat\phi 
\circ F^{-1}_E$ on all $e \in \partial\O$. Recall that, c.f. \eqref{rt-prop-0},
if $\vh \in \Vh^{k-1}_{RT}(\Eh)$, then $\vh\cdot\nh_{\hat{e}} \in 
\Qc^{k-1}(\eh)$ for all $\eh\subset \partial \Eh$. Then, using 
\eqref{RT-dof-1} and \eqref{rt-prop-3}, we have that
\begin{equation}\label{R-prop-1}
\forall \phi \in L^2(\partial\O), \quad 
\gnp[\phi - \Rc^{k-1}_h \phi]{\v\cdot\n}_{\partial\O} = 0 \quad 
\forall \v \in \Vh^{k-1}_{RT}(\Eh)
\end{equation}
and
\begin{equation}\label{R-prop-2}
\forall \v \in H^1(\O,\R^d), \quad 
\gnp[(\v - \Pi_{RT}^{k-1}\v)\cdot\n]{\Rc^{k-1}_h \phi}_{\partial\O} = 0 \quad
\phi \in L^2(\partial\O).
\end{equation}
The method is defined as follows: find $(\u_h,p_h) \in \V^k_h\times
W^{k-1}_h$, where $k\ge 1$, such that
	\begin{align}
	\inp[\Ki\u_h]{\v}_Q - \inp[p_h]{\dvrg\v} &=  - \gnp[\Rc^{k-1}_h g]{\v\cdot\n}_{\Gd}, \quad \v\in\V^k_h, \label{mfmfe-1}\\
	\inp[\dvrg\u_h]{w} &= \inp[f]{w}, \quad w\in W_h^{k-1}. \label{mfmfe-2}
	\end{align}
Following the terminology from
\cite{wheeler2006multipoint,ingram2010multipoint} we call the method
\eqref{mfmfe-1}-\eqref{mfmfe-2} a $k$-th order MFMFE method, due to
its relation to the MPFA scheme.
	
In order to prove that the method stated above has a unique solution, we first present several useful results.
	\begin{lemma}\label{lem-q-sim-qhat}
		If $E\in\Tc_h$ and $\q\in L^2(E,\R^d)$, then
		\begin{align}
		\|\q\|_E \sim h^{\frac{2-d}{2}}\|\qh\|_{\Eh}. \label{q-sim-qhat}
		\end{align}
	\end{lemma}
	\begin{proof}
		The statement of the lemma follows from \eqref{piola}:
		\begin{align*}
		\int_{E}\q\cdot\q\,d\x &= \int_{\Eh} \frac{1}{J_E}DF_E \qh\cdot \frac{1}{J_E}DF_E\qh\,J_E d\xbh, \\
		\int_{\Eh}\qh\cdot\qh\,d\xbh &= \int_{E} \frac{1}{J_{F_E^{-1}}}DF_E^{-1} \q\cdot \frac{1}{J_{F_E^{-1}}}DF_E^{-1}\q\, J_{F_E^{-1}}d\x,
		\end{align*}
		and bounds \eqref{mapping-bounds}.
	\end{proof}
	\begin{lemma} \label{inner-prod-norm}
The bilinear form $\inp[\Ki\q]{\v}_{Q}$ is an inner product on
$\V_h^k$ and $\inp[\Ki\q]{\q}_{Q}^{1/2}$ is a norm in
$\V_h^k$ equivalent to $\|\cdot\|$.
	\end{lemma}
	\begin{proof}
Let $\q\in \V_h^k$ be given on an element $E$ as
 $\q= \sum_{i=1}^{n_k}\sum_{j=1}^d q_{ij}\v_{ij}$. Using
\eqref{perm-bounds}, \eqref{mapping-bounds}, \eqref{quad-def-phys}, and
the basis property \eqref{basis}, we obtain
		\begin{align*}
		\inp[\Ki\q]{\q}_{Q,E} =  \sum_{i=1}^{n_k}J_E(\ph_i)w_k(i)\Ki(\p_i)\q(\p_i)\cdot\q(\p_i) \geq C h^d\sum_{i=1}^{n_k}\sum_{j=1}^d q_{ij}^2.
		\end{align*}
		On the other hand,
		\begin{align*}
		\|\q\|_{E}^2 = \inp[\sum_{i=1}^{n_k}\sum_{j=1}^d q_{ij}\v_{ij}]{\sum_{k=1}^{n_k}\sum_{l=1}^d q_{kl}\v_{kl}} \leq   C h^d\sum_{i=1}^{n_k}\sum_{j=1}^d q_{ij}^2.
		\end{align*}
		Hence,
		\begin{align}
		\inp[\Ki\q]{\q}_{Q} \geq C \|\q\|^2, \label{q_norm_ineq_1}
		\end{align}
and due to the linearity and symmetry, we conclude that
$\inp[\Ki\q]{\v}_Q$ is an inner product and
$\inp[\Ki\q]{\q}^{1/2}_{Q}$ is a norm in $\V_h^k$.  Using
\eqref{perm-bounds},\eqref{perm-mapped-bounds} \eqref{quad-rule-def},
\eqref{q-sim-qhat}, and the equivalence of norms on $\Eh$,
we obtain
		\begin{align}
		\inp[\Ki\q]{\q}_{Q,E} = \sum_{\bs{i} \in \Ic_k} w_k({\bs{i}})\Kci(\ph_{\bs{i}})\qh(\ph_{\bs{i}})\cdot\qh(\ph_{\bs{i}}) \leq Ch^{2-d}\|\qh\|^2_{\Eh} \leq C \|\q\|^2_E. \label{q_norm_ineq_2}
		\end{align}
Combining \eqref{q_norm_ineq_1} and \eqref{q_norm_ineq_2} results in the equivalence of norms
		 \begin{align}
		 c_0\|\q\| \leq \inp[\Ki\q]{\q}^{1/2}_{Q} \leq c_1\|\q\|. \label{norm-equiv}
		 \end{align}
	\end{proof}
We now proceed with the solvability of the method \eqref{mfmfe-1}-\eqref{mfmfe-2}.
	\begin{theorem}\label{well-posedness}
The $k$-th order MFMFE method \eqref{mfmfe-1}-\eqref{mfmfe-2} has a unique solution for any $k\ge 1$.
	\end{theorem}
	\begin{proof}
Since \eqref{mfmfe-1}-\eqref{mfmfe-2} is a square system, it is enough
to prove uniqueness of the solution. Letting $f=0,\,g=0$ and choosing
$\v=\u_h$ and $w=p_h$, one immediately obtains $\inp[\Ki\u_h]{\u_h}_Q
= 0$, which yields $\u_h = 0$ due to \eqref{norm-equiv}. Next, we use
the inf-sup condition \eqref{inf-sup} to obtain
		\begin{align*}
		\|p_h\| \leq C\sup\limits_{\q\in \V^k_h} 
\frac{\inp[\dvrg \q]{p_h}}{\|\q\|_{\dvr}} = \sup\limits_{\q\in \V^k_h} \frac{\inp[\Ki\u_h]{\q}_Q}{\|\q\|_{\dvr}} = 0
		\end{align*}
		and thus $p_h = 0$, which concludes the proof of the theorem.
	\end{proof}
	
\subsection{Reduction to a pressure system and its stencil.}	
In this section we describe how the MFMFE method reduces to a system
for the pressures by local velocity elimination. Recall that the DOFs
of $\Vh^k(\Eh)$ are chosen as the $d$ vector components at the
tensor-product Gauss-Lobatto quadrature points, see Figure
\ref{fig:2_1}.  As a result, in the velocity mass matrix obtained from
the bilinear form $(K^{-1} \u_h,\v)$, the $d$ DOFs associated with a
quadrature point in an element $E$ are completely decoupled from other
DOFs in $E$, see \eqref{local}. Due to the continuity of normal components
across facets, there are couplings with DOFs from neighboring
elements. We distinguish three types of velocity couplings. The first
involves localization of degrees of freedom around each vertex in the
grid. Only this type occurs in the lowest order case $k=1$, similar to
the previously developed lowest order MFMFE method
\cite{wheeler2006multipoint,ingram2010multipoint}. The number of DOFs
that are coupled around a vertex equals the number of facets $n_v$
that share the vertex. For example, on logically rectangular grids,
$n_v=12$ (faces) in 3d and $n_v=4$ (edges) in 2d. The second type of
coupling is around nodes located on facets, but not at vertices. In
2d, these are edge DOFs. The number of coupled DOFs is three - one
normal to the edge, which is continuous across the edge, and two
tangential to the edge, one from each of the two neighboring
elements. In 3d, there are two cases to consider for this type of
coupling. One case is for nodes located on faces, but not on edges. In
this case the number of coupled DOFs is five - one normal to the face,
which is continuous across the face, and four tangential to the face,
two from each of the two neighboring elements. The second case in 3d
is for nodes located on edges, but not at vertices. Let $n_e$ be the
number of elements that share the edge, which also equals the number
of faces that share the edge. In this case the number of coupled DOFs
is $2n_e$. These include $n_e$ DOFs normal to the $n_e$ faces, which
are continuous across the faces, and $n_e$ DOFs tangential to the
edge, one per each of the $n_e$ neighboring elements. For example, on
logically rectangular grids, $n_e = 4$, resulting in eight coupled
DOFs. Finally, the third type of coupling involves nodes interior to
the elements, in which case only the $d$ DOFs associated with the node
are coupled. 

Due to the localization of DOF interactions described above, the
velocity mass matrix obtained from the bilinear form $(K^{-1}
\u_h,\v)$, is block-diagonal with blocks associated with the
Gauss-Lobatto quadrature points. In particular, in 2d, there are $n_v
\times n_v$ blocks at vertices ($n_v$ is the number of neighboring
edges), $3 \times 3$ blocks at edge points, and $2 \times 2$ blocks at
interior points. In 3d, there are $n_v \times n_v$ blocks at vertices
($n_v$ is the number of neighboring faces), $2n_e \times 2n_e$ blocks
at edge points ($n_e$ is the number of neighboring elements), $5
\times 5$ blocks at face points, and $3 \times 3$ blocks at interior
points.

\begin{proposition}\label{reduct-prop-1}
The local matrices described above are symmetric and positive definite.
\end{proposition}
\begin{proof}
For any quadrature point, the local matrix is obtained by taking $\v =
\v_1,\dots,\v_m$ in \eqref{mfmfe-1}, where $\v_i$ are the velocity
basis functions associated with that point. We have
\begin{align*}
\inp[\Ki\u_h]{\v_i}_Q = \sum_{j=1}^m u_j \inp[\Ki\v_j]{\v_i} 
\equiv \sum_{j=1}^m a_{ij}u_j, \quad i=1,\dots,m.
\end{align*}
Using Lemma \ref{inner-prod-norm} we conclude that the matrix 
$M = \{a_{ij}\}$ is symmetric and positive definite.
\end{proof}
	
The block-diagonal structure of the velocity mass matrix allows for
local velocity elimination. In particular, solving the local linear
systems resulting from \eqref{mfmfe-1} allows us to express the
associated velocities in terms of the pressures from the neighboring
elements and boundary data. This implies that the method reduces the
saddle-point problem to an element-based pressure system.

\begin{lemma}
The pressure system resulting from \eqref{mfmfe-1}-\eqref{mfmfe-2}
using the procedure described above is symmetric and positive
definite.
\end{lemma}
\begin{proof}
The proof follows from the argument presented in Proposition 2.8 in
\cite{wheeler2006multipoint}. We present it here for the sake of
completeness. Denoting the bases of $\V^k_h$ and $W^{k-1}_h$ by
$\{\v_i\}$ and $\{w_i\}$, respectively, we obtain the saddle-point
type algebraic system arising from \eqref{mfmfe-1}-\eqref{mfmfe-2},
		\begin{align}
		\begin{pmatrix}
		A & B^T \\
		B & 0   
		\end{pmatrix}
		\begin{pmatrix}
		U \\
		P   
		\end{pmatrix}
		=
		\begin{pmatrix}
		G \\
		F   
		\end{pmatrix}, \label{saddle-point}
		\end{align}
where $A_{ij} = \inp[\Ki\v_i]{\v_j}_Q$ and $B^T_{ij} =
-\inp[\dvrg\v_i]{w_j}$. The matrix $A$ obtained by the above procedure
is symmetric and positive definite, as it is block diagonal with SPD
blocks associated with quadrature nodes shown in Proposition
\ref{reduct-prop-1}. The elimination of $U$ leads to a system for $P$
with a symmetric and positive semidefinite matrix $BA^{-1}B^T$. It follows
immediately from the proof of Theorem \ref{well-posedness} that $B^TP
= 0$ if and only if $P=0$. Therefore, $BA^{-1}B^T$ is positive
definite.
	\end{proof}
	
\begin{remark}
We note that while $\V^k_h$ has more DOFs than $\V^{k-1}_{RT,h}$ with
comparable accuracy, cf. Section~\ref{sec:vel}, the above reduction
technique allows for local elimination of all velocity DOFs, resulting
in a symmetric and positive definite system only for the pressure DOFs
in $W^{k-1}_h$. This is computationally more efficient than solving a
saddle point problem for the classical Raviart-Thomas MFE method in 
$\V^{k-1}_{RT,h}\times W^{k-1}_h$.
\end{remark}	

{\blue
\begin{remark} \label{tensor}
It was pointed out by an anonymous reviewer that it can be shown that
the matrix $B$ has a tensor product structure if it is formed using
the tensor product Gauss quadrature rule. This property can be
exploited for faster and low storage matrix-free assembly and
application of the matrix $B A^{-1} B^T$, resulting in further gain in
efficiency.  We thank the reviewer for noting this important property.
\end{remark}
}

	\section{Velocity error analysis.}\label{sec:vel}
	
Although the proposed schemes can be defined and are well posed on
general quadrilateral or hexahedra, for the convergence analysis we
need to impose a restriction on the element geometry. This is due to
the reduced approximation properties of the MFE spaces on arbitrary
shaped quadrilaterals or hexahedra that our new family of elements
inherits as well. The necessity of said restriction is confirmed by
the numerical computations.
We recall that, since the mapping $F_E$ is
trilinear in 3d, the faces of an element $E$ may be non-planar. We
will refer to the faces as generalized quadrilaterals. We recall the notation 
of $\r_i$, $i=1,\dots, 2^d$, and edges $\r_{ij} = \r_i - \r_j$ from 
Section~\ref{sec:map}.

\begin{definition}\label{def-h2-par2d}
A (generalized) quadrilateral with vertices $\r_i$, $i=1,\dots,4$,
is called an $h^2$-parallelogram if
$$
|\r_{34} - \r_{21}|_{\R^d} \le Ch^2.
$$
\end{definition}
The name follows the terminology from
\cite{ewing1999superconvergence,ingram2010multipoint}. Note that
elements of this type in 2d can be obtained by uniform refinements of a
general quadrilateral grid. It follows from \eqref{map-def-2d} that
$\frac{\partial^2F_E}{\partial\xh\partial\yh}$ is $\mathcal{O}(h^2)$
for $h^2$-parallelograms.
	
\begin{definition}\label{def-h2-par3d}
A hexahedral element is called an $h^2$-parallelepiped if all of its faces 
are $h^2$-parallelograms.
\end{definition}
	
\begin{definition}\label{def-reg-h2-par3d}
An $h^2$-parallelepiped with vertices $\r_i$, $i=1,\dots,8$, is called regular if
\begin{align*}
|(\r_{21} - \r_{34}) - (\r_{65} - \r_{78})|_{\R^3} \le Ch^3.
\end{align*}
\end{definition}
It is clear from \eqref{map-def-3d} that for
$h^2$-parallelepipeds,
$\frac{\partial^2F_E}{\partial\xh\partial\yh},\,\frac{\partial^2F_E}{\partial\yh\partial\zh}$
and $\frac{\partial^2F_E}{\partial\xh\partial\zh}$ are
$\mathcal{O}(h^2)$. Moreover, in case of regular $h^2$-parallelepipeds,
$\frac{\partial^3F_E}{\partial\xh\partial\yh\partial\zh}$ is
$\mathcal{O}(h^3)$.

We next present some bounds on the derivatives of the mapping $F_E$.
\begin{lemma}
Let $j\geq 0$. Then the bounds
\begin{equation}\label{J-bound}
|J_E|_{j,\infty,\Eh} \leq C h^{j+d}, \,\, j \leq \alpha, 
\mbox{ where } \alpha = 1 
\mbox{ in } 2d, \, \,
\alpha = 4 \mbox{ in } 3d, \,\, |J_E|_{j,\infty,\Eh} = 0,
\,\, j > \alpha,
\end{equation}
and
\begin{align}
		|DF_E|_{j,\infty,\Eh} \le\begin{cases}
		Ch^{j+1}, & j<d,\\
		0, & j\geq d
		\end{cases} ,\quad \left| \frac{1}{J_E}DF_E \right|_{j,\infty,\Eh} \le Ch^{j-d+1}, \quad |J_E DF_E^{-1}|_{j,\infty,\Eh} \leq \begin{cases}
		Ch^{j+d-1}, & j\leq d\\
		0, & j>d
		\end{cases} \label{map-deriv-bounds}
		\end{align}
hold if $E$ is an $h^2$-parallelogram or a regular
$h^2$-parallelepiped. Moreover, the estimates \eqref{map-deriv-bounds} 
hold for $j=0$ if $E$ is
a general quadrilateral or hexahedron and for $j=0,1$ if $E$ is an
$h^2$-parallelepiped.
	\end{lemma}
	\begin{proof}
We begin with the proof of \eqref{J-bound}. In 2d, \eqref{map-def-2d}
gives
\begin{align*}
DF_E = [\r_{21},\r_{41}] + [(\r_{34} - \r_{21})\yh,(\r_{34} - \r_{21})\xh],
\end{align*}
from which it can be shown easily that $J_E$ is a linear function satisfying
\eqref{J-bound}. In 3d, \eqref{map-def-3d} gives
\begin{equation}\label{DF}
\begin{split}
DF_E = [& \r_{21} + (\r_{34} - \r_{21})\yh + (\r_{65} - \r_{21})\zh
+ ((\r_{21} - \r_{34}) - (\r_{65} - \r_{78}))\yh\zh;\\
& \r_{41} + (\r_{34} - \r_{21})\xh + (\r_{85} - \r_{41})\zh
+ ((\r_{21} - \r_{34}) - (\r_{65} - \r_{78}))\xh\zh;\\
& \r_{51} + (\r_{65} - \r_{21})\xh + (\r_{85} - \r_{41})\yh
+ ((\r_{21} - \r_{34}) - (\r_{65} - \r_{78}))\xh\yh].
\end{split}
\end{equation}
It can be verified that $J_E$ is a polynomial of three variables of total 
power at most 4 with 
\begin{align}
(J_E)_{\xh\xh\xh} = (J_E)_{\yh\yh\yh} = (J_E)_{\zh\zh\zh} =0, 
\label{j-zero-partials}
\end{align}
and it can be written as
$
J_E = \sum_{0 \le r_1 + r_2 + r_3 \le 4} \alpha_{r_1r_2r_3} \xh^{r1}\yh^{r2}\zh^{r3},
$
where 
\begin{equation}\label{j-coeff-scaling}
|\alpha_{r_1r_2r_3}| \le C h^{r_1+r_2+r_3+3},
\end{equation}
from which \eqref{J-bound} follows immediately.

We proceed with the proof of \eqref{map-deriv-bounds}.  If $E$ is a
general quadrilateral or hexahedron, the bounds with
$j=0$ are stated in \eqref{mapping-bounds}. The estimates in 2d and
for $j = 1,\,2$ in 3d were shown in
\cite{wheeler2006multipoint,ewing1999superconvergence,ingram2010multipoint}.
We now focus on the case when $E$ is a regular $h^2$-parallelepiped
and $j>2$. Since {\blue $F_E$} is bilinear, 
$|DF_E|_{k,\infty,\Eh} =0, \, \forall k>2$, and \eqref{DF} gives
\begin{align}
|DF_E|_{k,\infty,\Eh} \leq Ch^{k+1}, \quad k=0,\,1,\,2. \label{df-scaling}
\end{align}
Therefore, it follows from the product rule that for any $j>2$,
		\begin{align}
		\left| \frac{1}{J_E} DF_E \right|_{j,\infty,\Eh} \leq C\left(\left| \frac{1}{J_E}\right|_{j,\infty,\Eh}|DF_E|_{0,\infty,\Eh}+ \left|\frac{1}{J_E}\right|_{j-1,\infty,\Eh}|DF_E|_{1,\infty,\Eh} +\left|\frac{1}{J_E}\right|_{j-2,\infty,\Eh}|DF_E|_{2,\infty,\Eh}\right). \label{j-inv-df-scaling}
		\end{align}
		We further compute the derivatives of $\displaystyle \frac{1}{J_E}$:
		\begin{align*}\allowdisplaybreaks
		\left(\frac{1}{J_E} \right)_{\xh} &= -\frac{1}{J_E^2}(J_E)_{\xh},  \quad \left(\frac{1}{J_E} \right)_{\xh\xh\xh} =-\frac{6}{J_E^4}(J_E)^3_{\xh} +\frac{6}{J_E^3}(J_E)_{\xh}(J_E)_{\xh\xh}  , \\
		\left(\frac{1}{J_E} \right)_{\xh\xh} &= \frac{2}{J_E^3}(J_E)^2_{\xh} -\frac{1}{J_E^2}(J_E)_{\xh\xh}, \quad\left(\frac{1}{J_E} \right)_{\xh\yh} = \frac{2}{J_E^3}(J_E)_{\xh}(J_E)_{\yh} -\frac{1}{J_E^2}(J_E)_{\xh\yh},\quad \\
		\left(\frac{1}{J_E} \right)_{\xh\xh\yh} &= -\frac{6}{J_E^4}(J_E)^2_{\xh} (J_E)_{\yh} +\frac{4}{J_E^3}(J_E)_{\xh}(J_E)_{\xh\yh} +\frac{2}{J_E^3} (J_E)_{\yh}(J_E)_{\xh\xh} -\frac{1}{J_E^2}(J_E)_{\xh\xh\yh}   \\
		\left(\frac{1}{J_E} \right)_{\xh\yh\zh} &=-\frac{6}{J_E^4}(J_E)_{\xh}(J_E)_{\yh}(J_E)_{\zh}+\frac{2}{J_E^3}(J_E)_{\xh\zh}(J_E)_{\yh}+\frac{2}{J_E^3}(J_E)_{\xh}(J_E)_{\yh\zh} +\frac{2}{J_E^3}(J_E)_{\zh}(J_E)_{\xh\yh}-\frac{1}{J_E^2}(J_E)_{\xh\yh\zh},\\
		\left(\frac{1}{J_E} \right)_{\xh\xh\yh\zh} &= \frac{24}{J_E^5}(J_E)^2_{\xh}(J_E)_{\yh}(J_E)_{\zh}-\frac{12}{J_E^4}(J_E)_{\xh}(J_E)_{\yh}(J_E)_{\xh\zh}-\frac{6}{J_E^4}(J_E)^2_{\xh}(J_E)_{\yh\zh} -\frac{12}{J_E^4}(J_E)_{\xh}(J_E)_{\zh}(J_E)_{\xh\yh}\\
		& +\frac{4}{J_E^3}(J_E)_{\xh\zh}(J_E)_{\xh\yh} +\frac{4}{J_E^3}(J_E)_{\xh}(J_E)_{\xh\yh\zh} - \frac{6}{J_E^4}(J_E)_{\zh}(J_E)_{\yh}(J_E)_{\xh\xh}+ \frac{2}{J_E^3}(J_E)_{\xh\xh}(J_E)_{\yh\zh} \\
		& + \frac{2}{J_E^3}(J_E)_{\yh}(J_E)_{\xh\xh\zh}+ \frac{2}{J_E^3}(J_E)_{\zh}(J_E)_{\xh\xh\yh} - \frac{1}{J_E^2}(J_E)_{\xh\xh\yh\zh}.
		\end{align*}
We note that due to \eqref{j-zero-partials} the higher order partial
derivatives will consist of the same partials that appear above, while
the power of $J_E$ in the denominator will continue to
grow. Therefore, it follows from \eqref{j-coeff-scaling} that
$\left|\frac{1}{J_E} \right|_{k,\infty,\Eh} \leq Ch^{k-3}$, which, combined
with \eqref{df-scaling} and \eqref{j-inv-df-scaling}, implies that
\begin{align*}
\left| \frac{1}{J_E} DF_E \right|_{j,\infty,\Eh} \leq C\left( h^{j-3}h + h^{j-4}h^2 + h^{j-5}h^3 \right) \leq Ch^{j-2}.
\end{align*}
To show the last inequality in \eqref{map-deriv-bounds}, we note that
using the cofactor formula for inverse of a matrix, one can verify
that $J_E DF_E^{-1}$ is of total degree 3, which implies that for
every $k>3$, $|J_E DF_E^{-1}|_{k,\infty,\Eh} = 0$. 
We also compute
		\begin{align*}
		((J_E DF_E^{-1})_{11})_{\xh\xh\yh} &= 2\big[(y_1-y_2)+(y_3-y_4)\big]\big[(z_5-z_6)+(z_7-z_8)+(z_2-z_1)+(z_4-z_3) \big] \\
		&+ 2\big[(z_1-z_2)+(z_3-z_4)\big]\big[(y_6-y_5)+(y_8-y_7)+(y_1-y_2)+(y_3-y_4)\big],
		\end{align*}
with similar expressions for the rest of partial derivatives. Therefore
$ |J_E DF_E^{-1}|_{3,\infty,\Eh} \leq Ch^5$.
\end{proof}
	
The above bounds allow us to control the norms of the velocity and permeability on the reference element. 
\begin{lemma}
For all $\q\in H^j(E)$, there exists a constant $C$ independent of $h$ such 
that the bound
\begin{align}
|\qh|_{j,\Eh} \le Ch^{j+\frac{d-2}{2}}\|\q\|_{j,E} \label{vel-bound}
\end{align}
holds for every $j\ge 0$ if $E$ is an $h^2$-parallelogram or regular
$h^2$-parallelepiped, for $j=0,1$ if $E$ is an $h^2$-parallelepiped
and for $j=0$ if $E$ is a general quadrilateral or
hexahedron. 
\end{lemma}
\begin{proof}
The result in 2d was shown in 
\cite{ewing1999superconvergence,wheeler2006multipoint}, while the 
cases $j=0,1,2$ in 3d
were proven in \cite{ingram2010multipoint}. It
then suffices to prove the case $j\ge 3$ for regular $h^2$-parallelepipeds. Let
		$$\tilde{\q} = \q\circ F_E(\xbh), \quad \qh = J_E DF^{-1}_E\tilde{\q}. $$
As it was shown in the previous lemma $|J_E DF_E^{-1}|_{4,\infty,\Eh} = 0$, 
hence \eqref{map-deriv-bounds} implies that for $r\ge 3$,
\begin{align}
		|\qh|_{r,\Eh} \le C\left( h^2|\tilde{\q}|_{r,\Eh} + h^3|\tilde{\q}|_{r-1,\Eh} +  h^4|\tilde{\q}|_{r-2,\Eh} + h^5|\tilde{\q}|_{r-3,\Eh} \right). \label{tmp2}
\end{align}
By change of variables and the chain rule, we have that 
$|\tilde{\q}|_{j,\Eh} \le Ch^{j-3/2}\|\q\|_{j,E}$, which, combined with 
\eqref{tmp2}, completes the proof.
	\end{proof}
	
	\begin{lemma}\label{perm-bound-lemma}
There exists a constant $C$ independent of $h$ such that the bound 
\begin{align}
&|\Kci|_{j,\infty,\Eh} \le Ch^{j-d+2}\|K^{-1}\|_{j,\infty,E}.\label{perm-bound}
\end{align}
holds with $j\ge 0$ on $h^2$-parallelograms and regular
$h^2$-parallelepipeds, with $j=0,1$ on $h^2$-parallelepipeds and with
$j=0$ on general quadrilaterals and hexahedra.
	\end{lemma}
	\begin{proof}
The above result with $j=0$ was already stated in
\eqref{perm-mapped-bounds}. Moreover, for $j=1,2$ \eqref{perm-bound}
was shown in \cite{wheeler2006multipoint,ingram2010multipoint}, so we
focus on the case $j\ge 3$ for $h^2$-parallelograms and regular
$h^2$-parallelepipeds.  By the use of a change of variables, the
chain rule, and \eqref{map-deriv-bounds}, it is easy to see that
\begin{align}
|\hat{K}^{-1}|_{j,\infty,\Eh} \le Ch^j|K^{-1}|_{j,\infty,E}. \label{k-scaling}
\end{align}
Using \eqref{map-deriv-bounds} and the definition of $\Kci$ given in 
\eqref{def-kci}, we have
\begin{align*}
|\Kci|_{j,\infty,\Eh} &\le C\sum_{\substack{0\le\alpha,\beta,\gamma\le j\\\alpha+\beta+\gamma = j}}|\frac{1}{J_E}DF_E|_{\alpha,\infty,\Eh}|\hat{K}^{-1}|_{\beta,\infty,\Eh}|DF_E|_{\gamma,\infty,\Eh} \\
&\le C\sum_{\substack{0\le\alpha,\beta,\gamma\le j\\\alpha+\beta+\gamma = j}}h^{\alpha-d+1}h^{\beta}h^{\gamma+1}\|K^{-1}\|_{j,\infty,E} \le Ch^{j-d+2}\|K^{-1}\|_{j,\infty,E},
\end{align*}
where we also used \eqref{k-scaling} for the second inequality.
	\end{proof}
	
	\begin{lemma}\label{app-prop-rtbubbles}
There exists a constant $C$ independent of $h$ such that on $h^2$-parallelograms and regular $h^2$-parallelepipeds
		\begin{align}
		\|\q - \Pi_*^k\q\| + \|\q - \Pi_{RT}^{k-1}\q\| &\le Ch^j \|\q\|_j, \label{app-prop-1}\\
		\|\q - \Pi_*^k\q\| &\le Ch^{j+1} \|\q\|_{j+1}, \label{app-prop-2}\\
		\|\dvrg\left(\q - \Pi_*^k\q\right)\| + \|\dvrg\left(\q - \Pi_{RT}^{k-1}\q\right)\| &\le Ch^j \|\dvrg\q\|_j,\label{app-prop-3} 
		\end{align}
		for $1\le j \le k$. Moreover, \eqref{app-prop-1} and \eqref{app-prop-3} also hold on $h^2$-parallelepipeds with $j=1$.
	\end{lemma}
	\begin{proof}
We present the proof for $\Pi_*^k$ only, as the argument for $\Pi_{RT}^{k-1}$ is similar. Using \eqref{q-sim-qhat} and \eqref{vel-bound}, 
we have 
		\begin{align*}
			\|\q - \Pi_*^k \q\|_E \le Ch^{\frac{d-2}{2}}\|\qh - \Pih_*^k \qh\|_{\Eh} \le Ch^{\frac{d-2}{2}}|\qh|_{j+1,\Eh} \le Ch^{j+1}\|\q\|_{j,E},
		\end{align*}
		where $1\le j \le k$. For the second inequality in the above, we used the fact that $\Pih_*^k$ preserves all polynomials of degree up to $k$, i.e., $\Pc^k(\Eh) \subset \Vh^k(\Eh)$, and applied the Bramble-Hilbert lemma \cite{ciarlet2002finite}. Summing over the elements completes the proof of the first two statements of the lemma.
		
		For the last inequality, it follows from \eqref{piola} that
		\begin{align}
			\int_E \left( \dvrg(\q - \Pi_*^k\q) \right)^2 d\x = \int_{\Eh} \frac{1}{J_E^2} \left( \nablah\cdot(\qh - \Pih_*^k\qh) \right)^2 J_E \,d\xbh \le Ch^{-d} |\nablah\cdot \qh|^2_{j,\Eh}, \label{tmp4}
		\end{align}
where we have used \eqref{mapping-bounds}, \eqref{rte-prop-div}, and
the Bramble-Hilbert lemma in the 
inequality. We also have
{\blue
		\begin{align}
		\begin{aligned}
|\nablah\cdot \qh|_{j,\Eh} = |J_E \widehat{\nabla\cdot\q}|_{j,\Eh} 
\le C\sum_{i=0}^{j} |J_E|_{i,\infty,\Eh} |\widehat{\nabla\cdot\q}|_{j-i,\Eh} 
\le C\sum_{i=0}^{j} h^{i+d}h^{j-i-\frac{d}{2}}|\dvrg\q|_{j-i,E}
\le Ch^{j+\frac{d}{2}}\|\dvrg\q\|_{j,E}, \label{tmp5}
		\end{aligned}
		\end{align}
		}
where we used \eqref{J-bound} and change of variables back to $E$ in the
second inequality. A combination of \eqref{tmp4} and \eqref{tmp5}, and a
summation over all elements completes the proof of \eqref{app-prop-3}.
	\end{proof}
	
	Let $\hat{\Qc}^{k-1}$ be the $L^2(\Eh)$-orthogonal projection onto $\Wh^{k-1}(\Eh)$, satisfying for any $\hat{\phi}\in L^2(\Eh)$,
	$$ \inp[\hat{\phi} - \hat{\Qc}^{k-1}\hat{\phi}]{\wh}_{\Eh} = 0\quad \forall\wh\in \Wh^{k-1}(\Eh). $$
	Let $\Qc_h^{k-1} : L^2(\O) \to W_h^{k-1}$ be the projection operator, satisfying for any $\phi\in L^2(\O)$,
	$$ \Qc_h^{k-1}\phi = \hat{\Qc}^{k-1}\hat{\phi}\circ F^{-1}_E \quad \mbox{on all }E. $$
	It follows from \eqref{rt-prop-00} that
	\begin{align}
	\inp[\phi - \Qc_h^{k-1}\phi]{\dvrg \v} = 0\quad \forall\v\in\V^k_h. \label{l2-prop-div}
	\end{align}
Using a scaling argument similar to \eqref{tmp4}-\eqref{tmp5}, one can
show that on $h^2$-parallelograms and regular $h^2$-parallelepipeds,
	\begin{align}
	\|\phi - \Qc_h^{k-1}\phi\| &\le Ch^j\|\phi\|_j, 
\quad 1\leq j \leq k. \label{l2-app-prop}
	\end{align}
Moreover, the above bound holds with $j=1$ on general quadrilaterals
and hexahedra and with $j=2$ on $h^2$-parallelepipeds. 

\begin{lemma}
For general quadrilaterals and hexahedra there exists a constant $C$ 
independent of $h$ such that for any finite element function $\varphi$
\begin{align}
\|\varphi\|_{j,E} \leq Ch^{-1}\|\varphi\|_{j-1,E}, \quad j=1,\dots,k.
\label{inverse-ineq}
\end{align}
\end{lemma}
\begin{proof} 
Let $\tilde{\varphi} =\varphi \circ F_E(\xh)$. Using \eqref{mapping-bounds},
we have
\begin{align*}
|\varphi|_{1,E} &\leq \|DF^{-1}_E\|_{0,\infty,E}\|J_E\|^{1/2}_{0,\infty,\Eh}|\tilde{\varphi}|_{1,\Eh} \leq C\|DF^{-1}_E\|_{0,\infty,E}\|J_E\|^{1/2}_{0,\infty,\Eh}\|\tilde{\varphi}\|_{\Eh} \\
			& \leq C \|DF^{-1}_E\|_{0,\infty,E}\|J_E\|^{1/2}_{0,\infty,\Eh} \|J_{F^{-1}_E}\|^{1/2}_{0,\infty,E}\|\varphi\|_{E} \leq Ch^{-1}h^{d/2}h^{-d/2}\|\varphi\|_{E} \leq Ch^{-1}\|\varphi\|_{E}.
\end{align*}
The general case follows by applying the above bound to any derivative of 
$\varphi$.
\end{proof}
We will make use of the following continuity bounds for the mixed projection 
operators $\Pi_*^k$ and $\Pi_{RT}^k$.
\begin{lemma}
There exists a constant $C$ independent of $h$ such that on $h^2$-parallelograms and regular $h^2$-parallelepipeds
\begin{align}
& \|\Pi_{*}^k\q\|_{j,E} \leq C\|\q\|_{j,E}, \quad j=1,\dots, k+1,\label{op-cont} \\
& \|\Pi^{k-1}_{RT}\q\|_{j,E} \leq C\|\q\|_{j,E}, \quad j=1,\dots, k, \label{rt-op-cont}
\end{align}
The above bounds also hold with $j=1$ on $h^2$-parallelepipeds. Furthermore,
on general quadrilaterals or hexahedra
\begin{equation}\label{div-bound}
\|\Pi_{*}^k\q\|_{\dvr,E} + \|\Pi^{k-1}_{RT}\q\|_{\dvr,E} \le C \|\q\|_{1,E}.
\end{equation}

		\end{lemma}
		\begin{proof}
It follows from \eqref{app-prop-1} and the triangle inequality that
\begin{align*}
\|\Pi_{*}^k\q\|_{0,E} \leq \|\q\|_{1,E}.
\end{align*}
Let $\Pc_E^j$ be the $L^2(E)$-projection onto $\Pc^j(E,\R^d)$. 
It is well known that
\cite{ciarlet2002finite} $\|\q - \Pc_E^j\q\|_E \leq C h^{j+1} \|\q\|_{j+1,E}$.
Using \eqref{inverse-ineq}, we have for any $j=1,\dots,k+1$,
\begin{align*}
|\Pi_{*}^k\q|_{j,E} &=  |\Pi_{*}^k\q - \Pc_E^{j-1}\q|_{j,E} 
\leq Ch^{-j} \|\Pi_{*}^k\q -\Pc_E^{j-1}\q\|_{0,E} 
\leq Ch^{-j}( \|\Pi_{*}^k\q - \q\|_{0,E} + \|\q - \Pc_E^{j-1}\q\|_{0,E})  
\leq C \|\q\|_j ,
\end{align*}
where we also used \eqref{app-prop-1}, \eqref{app-prop-2} and
\eqref{l2-app-prop}. This completes the proof of \eqref{op-cont}. The
proof of \eqref{rt-op-cont} is similar. The proof of \eqref{div-bound}
uses a scaling argument similar to \eqref{tmp4}-\eqref{tmp5} for the
divergence and a scaling argument using \eqref{vel-bound} for the
$L^2$-norm. Details can be found in Lemma 3.6 in
\cite{ingram2010multipoint}.
\end{proof}

\begin{remark}
For the rest of the paper, all results are stated for
$h^2$-parallelograms and regular $h^2$-parallelepipeds. We note that
the results also hold in 3d on $h^2$-parallelepipeds with $k=1$, except for 
the pressure superconvergence.
\end{remark}

In the next two lemmas we bound two terms arising in the
error analysis due to the use of the quadrature rule. We use the notation
$\varphi \in W^{k,\infty}_{\Tc_h}$ if 
$\varphi \in W^{k,\infty}(E) \,\, \forall \, E \in \Tc_h$ and 
$\|\varphi\|_{k,\infty,E}$ is uniformly bounded independently of $h$. 
\begin{lemma}
On $h^2$-parallelograms and regular $h^2$-parallelepipeds, if $\Ki \in
W^{k,\infty}_{\Tc_h}$, then there exists a constant $C$ independent of
$h$ such that for all $\v\in \V^k_h$,
\begin{align}
|\inp[K^{-1}\Pi_*^k \u]{\v-\Pi^{k-1}_{RT}\v}_Q| 
\leq Ch^k\|\u\|_k\|\v\|. \label{quad-error-1}
\end{align}
\end{lemma}
\begin{proof}
Let $\hat\Pc^k$ be the $L^2(\Eh)$-orthogonal projection onto $\Pc^k(\Eh,\R^d)$. 
For any element $E\in \Tc_h$, we have
\begin{align*}
\inp[K^{-1}\Pi_*^k \u]{\v-\Pi^{k-1}_{RT}\v}_{Q,E} 
&= \inp[\Kc^{-1}\Pih_*^k \hat{\u}]{\hat{\v}-\Pih^{k-1}_{RT}\hat{\v}}_{Q,\Eh} \\
&= \inp[\hat\Pc^{k-1}(\Kc^{-1}\Pih_*^k \hat{\u})]{\hat{\v}
-\Pih^{k-1}_{RT}\hat{\v}}_{Q,\Eh} 
+ \inp[\Kc^{-1}\Pih_*^k \hat{\u}-\hat\Pc^{k-1}(\Kc^{-1}\Pih_*^k \hat{\u})]{\hat{\v}-\Pih^{k-1}_{RT}\hat{\v}}_{Q,\Eh}.
\end{align*}
		The first term on right is equal to zero due to \eqref{exactness-quad}. For the second term we use Bramble-Hilbert lemma:
		\begin{align*}
		\left|\inp[\Kc^{-1}\Pih_*^k \hat{\u}-\hat\Pc^{k-1}(\Kc^{-1}\Pih_*^k \hat{\u})]{\hat{\v}-\Pih^{k-1}_{RT}\hat{\v}}_{Q,\Eh}\right|&\leq C |\Kc^{-1}\Pih_*^k \hat{\u}|_{k,\Eh}\|\hat{\v}-\Pih^{k-1}_{RT}\hat{\v}\|_{0,\Eh}.
		\end{align*}
		Using \eqref{perm-bound} and \eqref{vel-bound}, we obtain
		\begin{align*}
		|\Kc^{-1}\Pih_*^k \hat{\u}|_{k,\Eh} \leq C\sum_{i=0}^k |\Kc^{-1}|_{k-i,\infty,\Eh}|\Pih_*^k \hat{\u}|_{i,\Eh} &\leq C\sum_{i=0}^k h^{k-i-d+2}\|K^{-1}\|_{k-i,\infty,E}h^{i+(d-2)/2}\|\Pi_*^k \u\|_{i,E} \\
		&\leq Ch^{k-d/2+1} \|K^{-1}\|_{k,\infty,E}\|\Pi_*^k \u\|_{k,E}.
		\end{align*}
Therefore, using \eqref{vel-bound}, \eqref{op-cont} and \eqref{rt-cont}, 
we get
		\begin{align*}
		\left|\inp[\Kc^{-1}\Pih_*^k \hat{\u}-\hat\Pc^{k-1}(\Kc^{-1}\Pih_*^k \hat{\u})]{\hat{\v}-\Pih^{k-1}_{RT}\hat{\v}}_{Q,\Eh}\right|&
		\leq Ch^{k-d/2+1} \|K^{-1}\|_{k,\infty,E}\|\u\|_{k,E}h^{(d-2)/2}\|\v\|_{0,E}\\
		&\leq Ch^k\|K^{-1}\|_{k,\infty,E}\|\u\|_{k,E}\|\v\|_{0,E}.
		\end{align*}
		The proof is completed by summing over all elements.
	\end{proof}
\begin{lemma}
On $h^2$-parallelograms and regular $h^2$-parallelepipeds, if $\Ki \in
W^{k,\infty}_{\Tc_h},$ then there exists a constant $C$ independent of
mesh size such that for all $\q\in \V^k_h$ and $\v \in
\V_{RT,h}^{k-1}$
\begin{align}
|\sigma\inp[K^{-1}\q]{\v}| \leq C \sum_{E \in \Tc_h}
h^k \|\Ki\|_{k,\infty,E}\|\q\|_{k,E}\|\v\|_E. \label{quad-error-2}
\end{align}
	\end{lemma}
	\begin{proof}
For each $E\in \Tc_h$ we have
		\begin{align*}
		\sigma_E\inp[K^{-1}\q]{\v} = \sigma_{\Eh}\inp[\hat\Pc^{k-1}(\Kc^{-1}\qh)]{\vh} + \sigma_{\Eh}\inp[\Kc^{-1}\qh-\hat\Pc^{k-1}(\Kc^{-1}\qh)]{\vh}.
		\end{align*}	
The first term on the right is equal to zero, since the tensor-product
Gauss-Lobatto quadrature rule is exact for polynomials of degree up to
$2k-1$. Using the Bramble-Hilbert lemma, \eqref{perm-bound}
and \eqref{vel-bound}, we bound the second term as follows:
		\begin{align*}
		\left|\sigma_{\Eh}\inp[\Kc^{-1}\qh-\hat\Pc^{k-1}(\Kc^{-1}\qh)]{\vh}\right|& \leq C|\Kc^{-1}\qh|_{k,\Eh}\|\vh\|_{\Eh} \leq C\sum_{i=0}^k |\Kc^{-1}|_{k-i,\infty,\Eh}|\qh|_{i,\Eh} \|\vh\|_{\Eh} \\
		&\leq Ch^{k-d/2+1} \|K^{-1}\|_{k,\infty,E}\|\q\|_{k,E}h^{(d-2)/2}\|\v\|_{E} \\
		&\leq Ch^k\|K^{-1}\|_{k,\infty,E}\|\q\|_{k,E}\|\v\|_{E}.
		\end{align*}
		Summing over all $E\in \Tc_h$, we obtain \eqref{quad-error-2}.
	\end{proof}

	\subsection{Optimal convergence for the velocity.}
	We subtract the numerical method \eqref{mfmfe-1}-\eqref{mfmfe-2} from the variational formulation \eqref{weak-1}-\eqref{weak-2} to obtain the error equations:
	\begin{align}
	\inp[\Ki\u]{\v} - \inp[\Ki\u_h]{\v}_Q - \inp[p-p_h]{\dvrg\v} &=  - \gnp[g -\Rc_h^{k-1} g]{\v\cdot\n}_{\Gd} , && \v\in\V^k_h, \label{error-eq-1}\\
	\inp[\dvrg(\u-\u_h)]{w} &= 0, && w\in W^{k-1}_h. \label{error-eq-2} 
	\end{align}
	Note that due to \eqref{rte-prop-div}, it follows from \eqref{error-eq-2} that
	\begin{align}
		\dvrg (\Pi_*^k\u -\u_h) = 0. \label{error-div-vanish}
	\end{align}
	If we take $\v = \Pi_*^k \u - \u_h$ in \eqref{error-eq-1}, then 
	\begin{align}
	\inp[\Ki\u]{\Pi_*^k \u - \u_h} - \inp[\Ki\u_h]{\Pi_*^k \u - \u_h}_Q 
+ \gnp[g -\Rc_h^{k-1} g]{(\Pi_*^k \u - \u_h )\cdot\n}_{\Gd} &= 0 .
	\end{align}
	Let $\w \equiv \Pi_*^k \u - \u_h$ then an algebraic manipulation of the above gives 
	\begin{align*}
	\inp[\Ki \w]{\w}_Q = -\inp[\Ki\u]{\w}+\inp[\Ki\Pi_*^k\u]{\w}_Q  
- \gnp[g -\Rc_h^{k-1} g]{\w\cdot\n}_{\Gd} .
	\end{align*}
	Moreover, rewriting the right-hand side gives 
	\begin{align} 
	\inp[\Ki \w]{\w}_Q 
	&= -\inp[\Ki\u]{\w-\Pi^{k-1}_{RT}\w} - \gnp[g -\Rc_h^{k-1} g]{\w\cdot\n}_{\Gd} -\inp[\Ki(\u-\Pi_*^k\u)]{\Pi^{k-1}_{RT}\w}  \label{er-eq-1-2} \\
	& \quad -\inp[\Ki\Pi_*^k\u]{\Pi^{k-1}_{RT}\w} + \inp[\Ki\Pi_*^k\u]{\Pi^{k-1}_{RT}\w}_Q + \inp[\Ki\Pi_*^k\u]{\w-\Pi^{k-1}_{RT}\w}_Q. \nonumber
	\end{align} 
Testing \eqref{weak-1} with $\w - \Pi^{k-1}_{RT}\w $ and using that
$\dvrg\w = \dvrg\Pi^{k-1}_{RT}\w = 0$, see \eqref{error-div-vanish} 
and \eqref{div-preserv},
we can rewrite the first two terms in \eqref{er-eq-1-2} as
\begin{align*}
	-\inp[\Ki\u]{\w - \Pi^{k-1}_{RT}\w}  - \gnp[g -\Rc_h^{k-1} g]{\w\cdot\n}_{\Gd} = \gnp[g]{(\w - \Pi^{k-1}_{RT}\w)\cdot\n}_{\Gd}  - \gnp[g -\Rc_h^{k-1} g]{\w\cdot\n}_{\Gd}  = 0,
	\end{align*}
using that, due to \eqref{R-prop-1}--\eqref{R-prop-2}, 
$\gnp[\Rc^{k-1}_h g]{(\w - \Pi^{k-1}_{RT}\w)\cdot\n}_{\Gd} = 0$
and  $\gnp[g - \Rc^{k-1}_h g]{\Pi^{k-1}_{RT}\w \cdot\n}_{\Gd} = 0$.
For the third term on the right in \eqref{er-eq-1-2} 
we use \eqref{app-prop-1} and \eqref{rt-cont} to get
	\begin{align*}
	|\inp[\Ki(\u-\Pi_*^k\u)]{\Pi^{k-1}_{RT}\w}| \leq Ch^k\|\Ki\|_{0, \infty}\|\u\|_k\|\w\|.
	\end{align*}
	To bound the fourth and fifth terms on the right in \eqref{er-eq-1-2}, we use \eqref{quad-error-2}, \eqref{op-cont} and \eqref{rt-cont}:
	\begin{align*}
	|-\inp[\Ki\Pi_*^k\u]{\Pi^{k-1}_{RT}\w} + \inp[\Ki\Pi_*^k\u]{\Pi^{k-1}_{RT}\w}_Q| = |\s(\Ki\Pi_*^k\u,\Pi^{k-1}_{RT}\w )| \leq Ch^k\|\Ki\|_{k,\infty}\|\u\|_{k}\|\w\|.
	\end{align*}
	For the last term on the right in \eqref{er-eq-1-2} we use \eqref{quad-error-1}:
	\begin{align*}
	|\inp[\Ki\Pi_*^k\u]{\w-\Pi^{k-1}_{RT}\w}_Q| \leq Ch^k\|\Ki\|_{k,\infty}\|\u\|_k\|\w\|.
	\end{align*}
Combining the above bounds, we obtain from \eqref{er-eq-1-2} that
	\begin{align}
	\inp[\Ki(\Pi_*^k\u -\u_h)]{\Pi_*^k \u -\u_h}_Q   
\leq  Ch^k\|\Ki\|_{k,\infty}\|\u\|_k\|\Pi_*^k \u -\u_h\|,
	\end{align}
implying that
\begin{align}
\|\Pi_*^k\u -\u_h\| \leq Ch^k\|\Ki\|_{k,\infty}\|\u\|_k. \label{er-eq-1-4}
\end{align}
Bounds \eqref{er-eq-1-4} and \eqref{error-div-vanish}, together with 
\eqref{app-prop-1} and \eqref{app-prop-3}, result in the following theorem.
\begin{theorem}\label{th-velocity-error}
Assume that the partition $\Tc_h$ consists of $h^2$-parallelograms in
2d or regular $h^2$-parallelepipeds in 3d. If $\Ki \in
W^{k,\infty}_{\Tc_h}$, for the velocity $\u_h$ of the MFMFE method
\eqref{mfmfe-1}-\eqref{mfmfe-2}, there exists a constant $C$
independent of $h$ such that
\begin{align}
\|\u-\u_h\| &\leq Ch^k\|\u\|_k, \label{velocity-error-est-1} \\
\|\dvrg(\u-\u_h)\| &\leq Ch^k \|\dvrg \u\|_k. \label{velocity-error-est-2}
\end{align}
\end{theorem}

	\section{Error estimates for the pressure.}
	In this section we use a standard inf-sup argument to prove optimal convergence for the pressure. We also employ a duality argument to establish superconvergence for the pressure.
	\subsection{Optimal convergence for the pressure.}
	\begin{theorem}\label{th-pressure-error-1}
Assume that the partition $\Tc_h$ consists of $h^2$-parallelograms in
2d or regular $h^2$-parallelepipeds in 3d. If $\Ki \in
W^{k,\infty}_{\Tc_h}$, then for the pressure $p_h$ of the MFMFE method
\eqref{mfmfe-1}-\eqref{mfmfe-2}, there exists a constant $C$
independent of $h$ such that
\begin{align}
\|p-p_h\| &\leq Ch^k\left(\|\u\|_k +\|p\|_k\right). \label{pressure-error-est-1}
		\end{align}
	\end{theorem}
	\begin{proof}
We first note that the RT$_{k-1}$ spaces $\V^{k-1}_{RT,h} \times W^{k-1}_h$ on 
general quadrilaterals and hexahedra satisfy an inf-sup condition similar to
\eqref{inf-sup}. The proof is the same as the argument in 
Lemma~\ref{lemma-inf-sup}. Hence, using 
\eqref{error-eq-1} and \eqref{R-prop-1}, we obtain
\begin{align*}
\|\Qc^{k-1}_h p -p_h \| &\leq \frac{1}{\beta} 
\sup_{0\neq \v\in V^{k-1}_{RT,h}}\frac{\inp[\Qc^{k-1}_h p -p_h ]{\dvrg \v}}{\|\v\|_{\dvr}} \\
&= \frac{1}{\beta} \sup_{0\neq \v\in V^{k-1}_{RT,h}}\frac{\inp[\Ki(\Pi_*^k\u -\u_h)]{\v}_Q - \inp[\Ki(\Pi_*^k\u -\u)]{\v}+\s(\Ki\Pi_*^k\u,\v)}{\|\v\|_{\dvr}} \\
& \leq \frac{C}{\beta}h^k\|\Ki\|_{k,\infty}\|\u\|_k,
		\end{align*}
where we used \eqref{er-eq-1-4}, \eqref{app-prop-1}, \eqref{quad-error-2},
and \eqref{op-cont} in the last inequality.
The result then follows from \eqref{l2-app-prop} and the triangle inequality.
\end{proof}

	\subsection{Superconvergence of the pressure.}
In this subsection we prove superconvergence of the pressure, i.e., we
show that $\|\Qc_h^{k-1} p - p_h\|$ is $\mathcal{O}(h^{k+1})$
for the MFMFE method of order $k$. We also apply local postprocessing
to obtain an improved approximation $p_h^* \in W_h^k$ such that
$\|p - p_h^*\|$ is $\mathcal{O}(h^{k+1})$.

The following bound on the quadrature error will be used in the
superconvergence analysis.
	
	\begin{lemma}
		On $h^2$-parallelograms and regular $h^2$-parallelepipeds, if $\Ki\in W^{k+1,\infty}_{\Tc_h}$, then for all $\q\in \V^k_h$ and $\v\in \V^{0}_{RT,h}$, there exists a positive constant $C$ independent of $h$ such that
		\begin{align}
			|\sigma\inp[\Ki\q]{\v}| \le 
C \sum_{E\in\Tc_h}h^{k+1}\|K^{-1}\|_{k+1,\infty,E}\|\q\|_{k+1,E}\|\v\|_{1,E}.\label{super-q-error}
		\end{align}
	\end{lemma}
	\begin{proof}
For any element $E$ we have $\sigma_E\inp[\Ki\q]{\v} =
\hat{\sigma}_{\Eh}\inp[\Kci\qh]{\vh}$. Since the quadrature rule is
{\blue exact for polynomials of degree up to $2k-1$} and $k \ge 1$, 
then it is exact for polynomials of degree up to $k$. An application of 
the Bramble-Hilbert lemma implies
\begin{align*}
\left| \hat{\sigma}_{\Eh}\inp[\Kci\qh]{\vh} \right| \le C\bigg( \big[\sum_{i=0}^{k} |\Kci|_{i,\infty,\Eh} |\qh|_{k-i,\Eh} \big] |\vh|_{1,\Eh} 
+ \big[\sum_{i=0}^{k+1} |\Kci|_{i,\infty,\Eh} |\qh|_{k+1-i,\Eh} \big] \|\vh\|_{\Eh} \bigg),
\end{align*}
where we used that $\vh$ is linear.
Using \eqref{vel-bound} and \eqref{perm-bound} we obtain
		\begin{align*}
\sigma_E\inp[\Ki\q]{\v} \le Ch^{k+1}\|\Ki\|_{k+1,\infty,E} \|\q\|_{k+1,E} \|\v\|_{1,E}.
		\end{align*}
		Summation over all elements completes the proof.
	\end{proof}
	The following result establishes superconvergence of the pressure if the $H^2$-elliptic regularity which is defined below holds. 
	Let $\phi$ be the solution of
	\begin{align} \label{adjoint-equation}
		-\dvrg K\nabla\phi &= -(\Qc_h^{k-1}p - p_h) 
\quad \mbox{in }\O, \qquad \phi = 0 \quad \mbox{on } \partial \Omega .
	\end{align}
We say that this problem satisfies $H^{2}$-elliptic regularity if 
\begin{align}
\| K \nabla \phi \|_1 + \|\phi\|_{2} \le C\|\Qc_h^{k-1}p - p_h\| \label{ell-reg}
\end{align}
with constant $C$ which may depend on $K$ and $\Omega$ but is 
independent of $\phi$. Some sufficient conditions for \eqref{ell-reg} can be found in \cite{grisvard1985elliptic,lions1972non}. 
	In the proof of the theorem below, we follow the argument in \cite{Douglas-Roberts-1985} with appropriate modification 
	to deal with the quadrature terms.
	\begin{theorem} \label{thm:superconvergence}
Assume that the partition $\Tc_h$ consists of $h^2$-parallelograms in
2d or regular $h^2$-parallelepipeds in 3d. Assume also 
that $\Ki\in W^{k+1,\infty}_{\Tc_h}$, and that the
$H^2$-elliptic regularity \eqref{ell-reg} holds. Then, for the
pressure $p_h$ of the MFMFE method \eqref{mfmfe-1}-\eqref{mfmfe-2},
there exists a constant $C$ independent of $h$ such that
\begin{align}\label{eqn:superconvergence}
	\|\Qc_h^{k-1} p - p_h\| \le Ch^{k+1}(\|\u\|_{k} + \|\dvrg\u\|_k).
\end{align}
	\end{theorem}

	\begin{proof}
The proof makes use of a duality argument. Let $\phi$ be the
solution of \eqref{adjoint-equation}.  Denoting $-K \nabla \phi$ by
$\u^*$, $(\u^*, \phi)$ satisfy
\begin{align}
\inp[\Ki \u^*]{\v} - \inp[\phi]{\dvrg \v} &= 0, \quad \v \in H(\dvr;\O), 
\label{super-err-1}\\
\inp[\dvrg \u^*]{q} &= -\inp[\Qc_h^{k-1}p - p_h]{q}, \quad q \in L^2(\O). 
\label{super-err-2}
\end{align}
{\blue
Taking $\v = \u - \u_h$, $q = -(\Qc_h^{k-1} p - p_h)$ and adding
the two equations gives
		\begin{align}
		\begin{aligned}
\inp[\Ki \u^*]{\u - \u_h} - \inp[\phi]{\dvrg (\u - \u_h)} 
- \inp[\dvrg \u^*]{\Qc_h^{k-1} p - p_h} = \| \Qc_h^{k-1}p - p_h \|^2 .
\end{aligned} \label{super-aux-1}
\end{align}
Consider the discretization of
\eqref{super-err-1}--\eqref{super-err-2} as in
\eqref{mfmfe-1}--\eqref{mfmfe-2} and let $(\u_h^*, \phi_h^*)$ be the
solution of the discrete problem.  We now use the Galerkin orthogonality
\eqref{error-eq-1}--\eqref{error-eq-2} with $\v = \Pi_{RT}^{k-1}\u_h^*$
and $w = \Qc_h^{k-1} \phi$ to get
\begin{equation}\label{Gal-orth}
\inp[\Ki\u]{\Pi_{RT}^{k-1}\u_h^*} - \inp[\Ki\u_h]{\Pi_{RT}^{k-1}\u_h^*}_Q 
- \inp[\Qc_h^{k-1}p-p_h]{\dvrg\Pi_{RT}^{k-1}\u_h^*} 
- \inp[\dvrg(\u-\u_h)]{\Qc_h^{k-1} \phi} = 0,
\end{equation}
where we used that $(p - \Qc_h^{k-1} p, \dvrg \Pi_{RT}^{k-1}\u_h^*) = 0$ 
due to \eqref{l2-prop-div} and 
$\gnp[g -\Rc_h^{k-1} g]{\Pi_{RT}^{k-1}\u_h^*\cdot\n}_{\Gd} = 0$
due to \eqref{R-prop-1}. Subtracting \eqref{Gal-orth} from \eqref{super-aux-1} 
and using the symmetry of $(\Ki \cdot, \cdot)$ and $(\Ki \cdot, \cdot)_Q$
gives
		\begin{multline*}
			\inp[\Ki (\u^* - \Pi_{RT}^{k-1}\u_h^*)]{\u} - \inp[\Ki \u^*]{\u_h}  + \inp[\Ki  \Pi_{RT}^{k-1}\u_h^*]{\u_h}_Q \\
			- \inp[\phi- \Qc_h^{k-1} \phi]{\dvrg (\u - \u_h)} 
- \inp[\dvrg (\u^*- \Pi_{RT}^{k-1} \u_h^*)]{\Qc_h^{k-1}p - p_h} 			= \| \Qc_h^{k-1}p - p_h \|^2 .
		\end{multline*}
Since $\dvrg \Pi_{RT}^{k-1}\u_h^* = \dvrg \u_h^*$, and $\inp[\dvrg
  (\u^*- \u_h^*)]{q} = 0$ holds for all $q \in W_h^{k-1}$ from the
definition of $\u_h^*$, the last term in the left-hand side
vanishes. Therefore we have
		\begin{align}
			\inp[\Ki (\u^* - \Pi_{RT}^{k-1}\u_h^*)]{\u - \u_h} - \sigma\inp[\Ki \Pi_{RT}^{k-1}\u_h^*]{\u_h}  - \inp[\phi- \Qc_h^{k-1} \phi]{\dvrg (\u - \u_h)} = \| \Qc_h^{k-1}p - p_h \|^2 . \label{duality-equation}
		\end{align}
		}
		with $\sigma\inp[\Ki \Pi_{RT}^{k-1}\u_h^*]{\u_h} = \inp[\Ki \Pi_{RT}^{k-1}\u_h^*]{\u_h} - \inp[\Ki \Pi_{RT}^{k-1}\u_h^*]{\u_h}_Q $. 
Observe that the difference of \eqref{super-err-1} and its discrete 
counterpart gives 
		\begin{align*}
		\inp[\Ki \u^*]{\Pi_{RT}^{k-1} \u - \u_h} - \inp[\Ki \u_h^*]{\Pi_{RT}^{k-1} \u - \u_h}_Q = 0,
		\end{align*}
because $\dvrg (\Pi_{RT}^{k-1} \u - \u_h) = 0$. From this we obtain 
		\begin{align*}
		\sigma\inp[\Ki \Pi_{RT}^{k-1} \u_h^*]{ \u_h} &= \sigma\inp[\Ki \Pi_{RT}^{k-1} \u_h^*]{ \Pi_{RT}^{k-1} \u} - \sigma\inp[\Ki \Pi_{RT}^{k-1} \u_h^*]{ \Pi_{RT}^{k-1} \u - \u_h } \\
		&= \sigma\inp[\Ki \Pi_{RT}^{k-1} \u_h^*]{ \Pi_{RT}^{k-1} \u} - \inp[\Ki \Pi_{RT}^{k-1} \u_h^*]{ \Pi_{RT}^{k-1} \u - \u_h } + \inp[\Ki \Pi_{RT}^{k-1} \u_h^*]{ \Pi_{RT}^{k-1} \u - \u_h }_Q \\
		&= \sigma\inp[\Ki \Pi_{RT}^{k-1} \u_h^*]{ \Pi_{RT}^{k-1} \u} + \inp[\Ki (\u^*-\Pi_{RT}^{k-1} \u_h^*)]{ \Pi_{RT}^{k-1} \u - \u_h } \\
		&\quad - \inp[\Ki (\u_h^* - \Pi_{RT}^{k-1} \u_h^*)]{\Pi_{RT}^{k-1} \u - \u_h}_Q,
		\end{align*}
and we can rewrite \eqref{duality-equation} further as
\begin{align}
\begin{aligned}
& \inp[\Ki (\u^* - \Pi_{RT}^{k-1}\u_h^*)]{\u - \Pi_{RT}^{k-1} \u} 
+ \inp[\Ki (\u_h^* - \Pi_{RT}^{k-1} \u_h^*)]{\Pi_{RT}^{k-1} \u - \u_h}_Q  \\
& \qquad\qquad 
- \sigma\inp[\Ki \Pi_{RT}^{k-1} \u_h^*]{ \Pi_{RT}^{k-1} \u} 
- \inp[\phi- \Qc_h^{k-1} \phi]{\dvrg (\u - \u_h)} = \| \Qc_h^{k-1}p - p_h \|^2 . 
\label{super-aux-2}
\end{aligned}
\end{align}
We will show that the terms on left above can be bounded as follows:
\begin{align}
\label{duality-estm-1} |\inp[\Ki (\u^* - \Pi_{RT}^{k-1}\u_h^*)]{\u - \Pi_{RT}^{k-1} \u}| 
&\le Ch^{k+1} \| \Qc_h^{k-1}p - p_h \| \| \u \|_k,  \\
\label{duality-estm-2} |\inp[\Ki (\u_h^* - \Pi_{RT}^{k-1} \u_h^*)]{\Pi_{RT}^{k-1} \u - \u_h}_Q| 
&\le Ch^{k+1} \| \Qc_h^{k-1}p - p_h \| \| \u \|_k , \\
\label{duality-estm-3} | \,\sigma\inp[\Ki \Pi_{RT}^{k-1} \u_h^*]{\Pi_{RT}^{k-1}\u}| &\le Ch^{k+1} \| \Qc_h^{k-1}p - p_h \| \| \u \|_k, \\
\label{duality-estm-4} | \inp[\phi- \Qc_h^{k-1} \phi]{\dvrg (\u - \u_h)} | &\le Ch^{k+1} \| \Qc_h^{k-1}p - p_h \| \| \dvrg \u \|_k,
\end{align}
which, combined with \eqref{super-aux-2}, imply the statement of the Theorem.
For \eqref{duality-estm-1}, we note that
\begin{align}
\begin{aligned}
\|\u^* - \Pi_{RT}^{k-1}\u_h^*\| & \le \|\u^* - \Pi_{RT}^{k-1}\u^*\| 
+ \|\Pi_{RT}^{k-1}(\Pi_{RT}^{k-1}\u^* - \u_h^*)\|
\le \|\u^* - \Pi_{RT}^{k-1}\u^*\| + 
C \|\Pi_{RT}^{k-1}\u^* - \u_h^*\| \\
& \le \|\u^* - \Pi_{RT}^{k-1}\u^*\| 
+ C(\|\Pi_{RT}^{k-1}\u^* - \u^*\| + \|\u^* - \u_h^*\|)
\le C h \|\u^*\|_1,
\label{super-aux-3}
\end{aligned}
\end{align}
where we used \eqref{rt-cont}, \eqref{app-prop-1}, and a bound for the
discretization error 
\begin{equation}\label{u-h-*-error}
\|\u^* - \u_h^*\| \le C h\|\u^*\|_1, 
\end{equation}
which is obtained in a manner similar to the velocity error estimate
\eqref{velocity-error-est-1}. Bound \eqref{duality-estm-1} 
follows from the use of the Cauchy--Schwarz inequality, \eqref{super-aux-3},
\eqref{app-prop-1}, and \eqref{ell-reg}. Bound \eqref{duality-estm-2}
is obtained in a similar way, by adding and subtracting $\u^*$ in the first
component and $\u$ in the second component, and using \eqref{u-h-*-error},
\eqref{super-aux-3}, \eqref{app-prop-1}, \eqref{velocity-error-est-1}, and
\eqref{ell-reg}. Bound \eqref{duality-estm-3} follows from 
\begin{align*}
| \sigma\inp[\Ki \Pi_{RT}^{k-1}\u_h^*]{\Pi_{RT}^{k-1}\u} | 
&\le 
|\sigma\inp[\Ki (\Pi_{RT}^{k-1}\u_h^* - \Pi_{RT}^0 \u^*)]{\Pi_{RT}^{k-1}\u} | 
+ | \sigma\inp[\Ki \Pi_{RT}^0 \u^* ]{\Pi_{RT}^{k-1}\u} | \\
& \le C (h^k \|\u \|_k \|\Pi_{RT}^{k-1}\u_h^* - \Pi_{RT}^0 \u^*\|
+ h^{k+1}\|\u\|_k \|\u^*\|_1)
\le C h^{k+1} \| \Qc_h^{k-1}p - p_h \| \| \u \|_k,
\end{align*}
where we used \eqref{quad-error-2}, \eqref{super-q-error}, 
\eqref{rt-op-cont}, \eqref{super-aux-3}, \eqref{app-prop-1}, 
and \eqref{ell-reg}. Finally, \eqref{duality-estm-4} follows from 
\eqref{l2-app-prop}, \eqref{velocity-error-est-2}, and \eqref{ell-reg}.
\end{proof}

Using the above result we can easily show superconvergence of the
pressure at the Gauss points. For an element $E$, let $|||\cdot|||_E$
denote the discrete $L^2(E)$-norm computed by mapping to the reference
element $\Eh$ and applying the tensor-product Gauss quadrature rule
with $k$ points in each variable. It is easy to see that $|||w|||_E =
\|w\|_E$ for $w \in W_h^{k-1}(E)$. Assuming continuous pressure $p|_E$, 
let $p^I|_E \in W_h^{k-1}(E)$ be 
the Lagrange interpolant of $p|_E$ at the $k^d$ Gauss points. It is shown in 
\cite[Lemma 4.3]{ELW} that
\begin{equation}\label{qp-pi}
\|\Qc_h^{k-1} p - p^I\| \le C h^{k+1} \|p\|_{k+1}.
\end{equation}
We now have
\begin{align*}
|||p - p_h||| = |||p^I - p_h||| = \|p^I -p_h\| \le \|p^I - \Qc_h^{k-1} p\|
+ \|\Qc_h^{k-1} p - p_h\| \le C h^{k+1} (\|\u\|_k + \|\nabla\cdot \u\|_k + \|p\|_{k+1}),
\end{align*}
using \eqref{qp-pi} and \eqref{eqn:superconvergence}.

We next show that the above superconvergence result for $\|
\Qc_h^{k-1}p - p_h \|$ can be used to compute a higher order
approximation to the pressure $p$ in the $L^2(\O)$-norm, 
using a variant of the local postprocessing proposed in \cite{Stenberg-91}. 
The postprocessing idea is also utilized for {\em a posteriori} 
error estimation (see e.g., \cite{Lovadina-Stenberg-06}).
Let $\tilde{W}_h^{k}$ be the $L^2$-orthogonal complement of 
$W_h^{0}$ in $W_h^{k}$. We now define $p_h^* \in W_h^k$ by  
	\begin{align}
\label{post-process1} \Qc_h^{0} p_h^* &= \Qc_h^0 p_h , \\
\label{post-process2} (\nabla p_h^*, \nabla q)_E &= - (\Ki \u_h, \nabla q)_E, \qquad q \in \tilde{W}_h^k (E), \forall E \in \Tc_h .
	\end{align}

	\begin{theorem}
Under the assumption of Theorem~\ref{thm:superconvergence}, 
there exists a constant $C$ independent of $h$ such that
\begin{align}
\| p - p_h^* \| \le Ch^{k+1} ( \|\u\|_{k} + \|\dvrg\u\|_k + \| p \|_{k+1}) .
\end{align}
	\end{theorem}
	\begin{proof}
Let $\tilde{\Qc}_h^k$ be the $L^2$ orthogonal projection onto
$\tilde{W}_h^k$.  By the triangle inequality it is enough to estimate
$\| \Qc_h^k p - p_h^* \|$. Let $\tilde{p}_h := p_h^* - \Qc_h^0 p_h$. 
Considering the decomposition 
$\Qc_h^k p - p_h^* = (\Qc_h^{0} p - \Qc_h^0 p_h) + (\tilde{\Qc}_h^k p - \tilde{p}_h)$, it is sufficient to estimate 
		$\| \tilde{\Qc}_h^k p - \tilde{p}_h \|$ by Theorem~\ref{thm:superconvergence}. Recalling that {\blue $\nabla p = - K^{-1} \u$}, we have 
		\begin{align*}
		(\nabla_h (p - p_h^*), \nabla_h q) = - (\Ki (\u - \u_h), \nabla_h q), \qquad \forall q \in \tilde{W}_h^k,
		\end{align*}
		where $\nabla_h$ is the element-wise gradient. From $p-p_h^* = (p - \Qc_h^k p) + (\Qc_h^{0} p - \Qc_h^0 p_h) + (\tilde{\Qc}_h^k p - \tilde{p}_h)$ 
		and by taking $q = \tilde{\Qc}_h^k p - \tilde{p}_h$ in the above equation, we get 
		\begin{align*}
		\| \nabla_h (\tilde{\Qc}_h^k p - \tilde{p}_h) \| &\le \| \nabla_h (p - \Qc_h^k p) \| + \| \Ki (\u - \u_h) \| 
		\le C h^k (\| p \|_k  + \| \u \|_k),
		\end{align*}
		where we used the Bramble--Hilbert lemma, an inverse estimate, 
and \eqref{velocity-error-est-1}.
		Since $W_h^{0}$ is the space of element-wise constants on $\Tc_h$, $\tilde{\Qc}_h^k p - \tilde{p}_h$ is orthogonal to element-wise constants. 
		Then the element-wise Friedrichs' inequality yields 
$\| \tilde{\Qc}_h^k p - \tilde{p}_h \|_E \le C h_E \| \nabla_h (\tilde{\Qc}_h^k p - \tilde{p}_h) \|_E$
for all $E \in \mathcal{T}_h$. The conclusion follows by combining this and the above inequality. 
	\end{proof}

\begin{remark}
Instead of the postprocessing 
\eqref{post-process1}-\eqref{post-process2}, one may use the 
postprocessing defined in \cite{Stenberg-91} and obtain a numerical
pressure that is convergent of order $\mathcal{O}(h^{k+1})$. The error
analysis is almost the same as the above.
\end{remark}

\section{Numerical results.}
In this section we present numerical experiments on
quadrilateral and hexahedral grids that validate the theoretical
results in the previous sections. {\blue The method has been implemented in
the finite element library deal.II \cite{dealii}. The code is available in 
the deal.II code gallery \cite{hMFMFE_code}}. In the first example we test the
method on a sequence of meshes obtained by a uniform isotropic
refinement of {\blue{an initial quadrilateral partition of the unit square}}. The boundary
conditions are chosen to be of Dirichlet type for simplicity. The test
case is constructed with the full permeability tensor coefficient
	\begin{align*}
	K = 
	\begin{pmatrix}
	(x+1)^2 + y^2 & \sin{(xy)} \\ 
	\sin{(xy)}	  & (x+1)^2
	\end{pmatrix}, 
	\end{align*}
	and the analytical solution
	\begin{align*}
	p = x^3y^4 + x^2 + \sin(xy)\cos(xy).
	\end{align*}
	%
  	\begin{figure}
  		\centering
  		\begin{subfigure}[b]{0.3\textwidth}
  			\includegraphics[width=\textwidth]{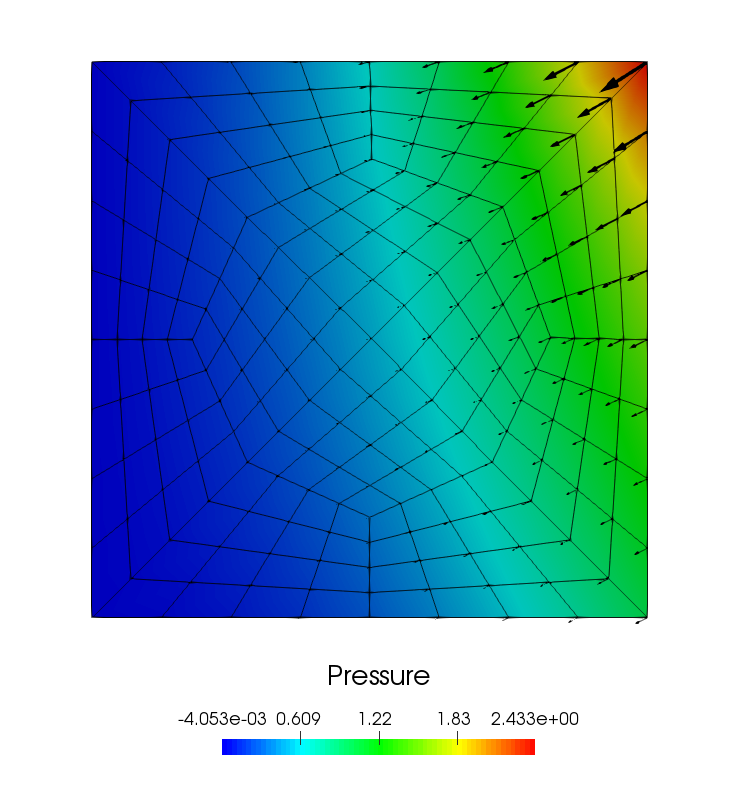}
  		\end{subfigure}
  		\begin{subfigure}[b]{0.3\textwidth}
  			\includegraphics[width=\textwidth]{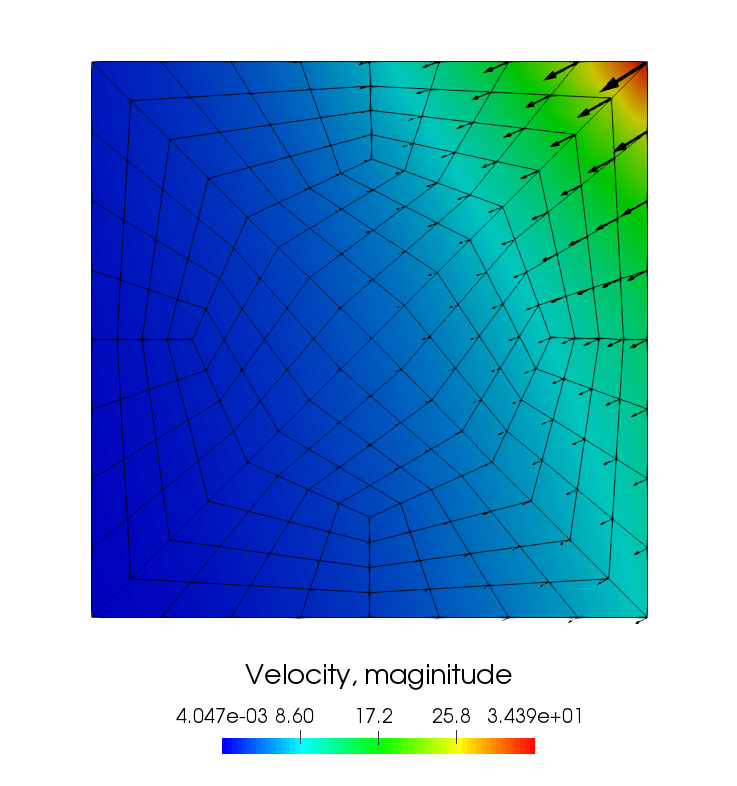}
  		\end{subfigure}
  		\caption{Computed solution for Example 1 
on the third level of refinement}\label{fig:4_1}
  	\end{figure}
	%

The computed pressure solution on the third level of refinement is
shown in Figure \ref{fig:4_1} (left), where the colors represent the
pressure values and the arrows represent the velocity
vectors. Similarly, Figure \ref{fig:4_1} (right) shows the velocity
solution, where colors represent the velocity magnitude. The numerical
relative errors and convergence rates are obtained on a sequence of
six mesh refinements and are reported in Table \ref{tab:4_1} for the
MFMFE methods of order $k=2,\,3,\,4$. We note that in all cases we see
the predicted convergence rate of order $\mathcal{O}(h^k)$ for all
variables in their natural norms, as well as superconvergence of
the pressures at the Gauss points, i.e., $|||p - p_h|||$ is of order
$\mathcal{O}(h^{k+1})$. We also observe
$\mathcal{O}(h^{k+1})$ convergence for the postprocessed pressure.
We note that the deterioration of the convergence rate of the divergence and
the superconvergence rate of the pressure for the $4$-th order method on the
finest grid is due to the fact that these errors
are very small and roundoff errors start having a noticeable effect.

	\begin{table}
		\begin{center}
			\begin{tabular}{r|cc|cc|cc|cc|cc} \hline
				\multicolumn{11}{c}{$k = 2$} \\ \hline
				& 
				\multicolumn{2}{c|}{$ \|\u - \u_h\| $} & 
				\multicolumn{2}{c|}{$ \|\nabla\cdot(\u - \u_h)\| $} &  
				\multicolumn{2}{c|}{$ \|p - p_h\| $} &
				\multicolumn{2}{c|}{$ |||p - p_h||| $} &
				\multicolumn{2}{c}{$ \|p - p_h^*\| $}\\ 
				$h$ & error & rate & error & rate & error & rate & error & rate & error & rate \\ \hline
				1/3	&	8.80E-02	&	--	&	1.46E-01	&	--	&	3.20E-02	&	--	&	5.80E-03	&	--	&	1.19E-02	&	--	\\	
				1/6	&	2.36E-02	&	1.9	&	3.74E-02	&	2.0	&	7.90E-03	&	2.0	&	7.73E-04	&	2.9	&	1.42E-03	&	3.1	\\	
				1/12	&	6.01E-03	&	2.0	&	9.41E-03	&	2.0	&	1.98E-03	&	2.0	&	1.18E-04	&	2.7	&	1.66E-04	&	3.1	\\	
				1/24	&	1.50E-03	&	2.0	&	2.36E-03	&	2.0	&	4.96E-04	&	2.0	&	1.70E-05	&	2.8	&	1.94E-05	&	3.1	\\	
				1/48	&	3.74E-04	&	2.0	&	5.89E-04	&	2.0	&	1.24E-04	&	2.0	&	2.30E-06	&	2.9	&	2.29E-06	&	3.1	\\	
				1/96	&	9.31E-05	&	2.0	&	1.47E-04	&	2.0	&	3.10E-05	&	2.0	&	2.99E-07	&	2.9	&	2.78E-07	&	3.1	\\	\hline\hline
				\multicolumn{11}{c}{$k = 3$} \\ \hline
				& 
				\multicolumn{2}{c|}{$ \|\u - \u_h\| $} & 
				\multicolumn{2}{c|}{$ \|\nabla\cdot(\u - \u_h)\| $} &  
				\multicolumn{2}{c|}{$ \|p - p_h\| $} &
				\multicolumn{2}{c|}{$ |||p - p_h||| $} &
				\multicolumn{2}{c}{$ \|p - p_h^*\| $}\\ 
				$h$ & error & rate & error & rate & error & rate & error & rate & error & rate \\ \hline
				1/3	&	1.35E-02	&	--	&	1.96E-02	&	--	&	3.16E-03	&	--	&	4.36E-04	&	--	&	1.03E-03	&	--	\\	
				1/6	&	1.69E-03	&	3.0	&	2.44E-03	&	3.0	&	3.95E-04	&	3.0	&	3.33E-05	&	3.7	&	5.33E-05	&	4.3	\\	
				1/12	&	2.09E-04	&	3.0	&	3.04E-04	&	3.0	&	4.95E-05	&	3.0	&	2.48E-06	&	3.8	&	2.79E-06	&	4.3	\\	
				1/24	&	2.59E-05	&	3.0	&	3.80E-05	&	3.0	&	6.19E-06	&	3.0	&	1.74E-07	&	3.8	&	1.55E-07	&	4.2	\\	
				1/48	&	3.22E-06	&	3.0	&	4.75E-06	&	3.0	&	7.73E-07	&	3.0	&	1.17E-08	&	3.9	&	9.04E-09	&	4.1	\\	
				1/96	&	4.02E-07	&	3.0	&	5.93E-07	&	3.0	&	9.67E-08	&	3.0	&	7.57E-10	&	4.0	&	5.44E-10	&	4.1	\\	\hline\hline
				\multicolumn{11}{c}{$k = 4$} \\ \hline
				& 
				\multicolumn{2}{c|}{$ \|\u - \u_h\| $} & 
				\multicolumn{2}{c|}{$ \|\nabla\cdot(\u - \u_h)\| $} &  
				\multicolumn{2}{c|}{$ \|p - p_h\| $} &
				\multicolumn{2}{c|}{$ |||p - p_h||| $} &
				\multicolumn{2}{c}{$ \|p - p_h^*\| $}\\ 
				$h$ & error & rate & error & rate & error & rate & error & rate & error & rate \\ \hline
				1/3	&	1.13E-03	&	--	&	1.52E-03	&	--	&	2.46E-04	&	--	&	2.83E-05	&	--	&	5.17E-05	&	--	\\	
				1/6	&	6.84E-05	&	4.1	&	9.24E-05	&	4.0	&	1.52E-05	&	4.0	&	1.00E-06	&	4.8	&	1.26E-06	&	5.4	\\	
				1/12	&	4.20E-06	&	4.0	&	5.74E-06	&	4.0	&	9.50E-07	&	4.0	&	3.55E-08	&	4.8	&	3.20E-08	&	5.3	\\	
				1/24	&	2.59E-07	&	4.0	&	3.58E-07	&	4.0	&	5.94E-08	&	4.0	&	1.20E-09	&	4.9	&	8.74E-10	&	5.2	\\	
				1/48	&	1.61E-08	&	4.0	&	2.25E-08	&	4.0	&	3.71E-09	&	4.0	&	3.98E-11	&	4.9	&	2.59E-11	&	5.1	\\	
				1/96	&	1.00E-09	&	4.0	&	4.96E-09	&	2.2	&	2.32E-10	&	4.0	&	8.78E-12	&	2.2	&	8.72E-12	&	1.6	\\	\hline
			\end{tabular}
		\end{center}
		\caption{Relative errors and convergence rates for Example 1.} \label{tab:4_1}
	\end{table}

In the second example, we focus on a 3d case. We let $K$ be a full
 permeability tensor with variable coefficients
	\begin{align*}
	K = 
	\begin{pmatrix}
	x^2 + (y+2)^2 & 0 		 & \cos(xy) \\
	0			  & z^2 + 2  & \sin(xy) \\
	\cos(xy)	  & \sin(xy) & (y+3)^2,
	\end{pmatrix}
	\end{align*}
	and solve the problem with Dirichlet boundary conditions and the analytical pressure solution chosen as follows
	\begin{align*}
	p = x^4y^3 + x^2 + yz^2 +\cos(xy) + \sin(z).
	\end{align*}
The initial computational domain is obtained as a smooth map of the
unit cube, i.e., we start with a $4\times 4\times 4$ unit cube mesh
and then apply the following transformation to its points
	\begin{align*}
	x &= \xh + 0.03\cos(3\pi\xh)\cos(3\pi\yh)\cos(3\pi\zh)\\
	y &= \yh - 0.04\cos(3\pi\xh)\cos(3\pi\yh)\cos(3\pi\zh)\\
	z &= \zh + 0.05\cos(3\pi\xh)\cos(3\pi\yh)\cos(3\pi\zh).
	\end{align*}
The sequence of meshes on which we perform the convergence study is
then obtained by a series of uniform refinements of the initial grid,
described above. Figure \ref{fig:4_2} (left) presents the pressure
solution, computed on the third level of refinement, where the colors
represent the pressure values and the arrows depict the velocity
vectors. The velocity magnitude is also shown in Figure \ref{fig:4_2}
(right). The computed numerical errors and convergence rates shown in
Table \ref{tab:4_2} once again confirm the theoretical results from
the error analysis section. We see the optimal $\mathcal{O}(h^k)$
order of convergence for all variables, and also
$\mathcal{O}(h^{k+1})$ superconvergence for the pressure.

	%
 	\begin{figure}
 		\centering
 		\begin{subfigure}[b]{0.35\textwidth}
 			\includegraphics[width=\textwidth]{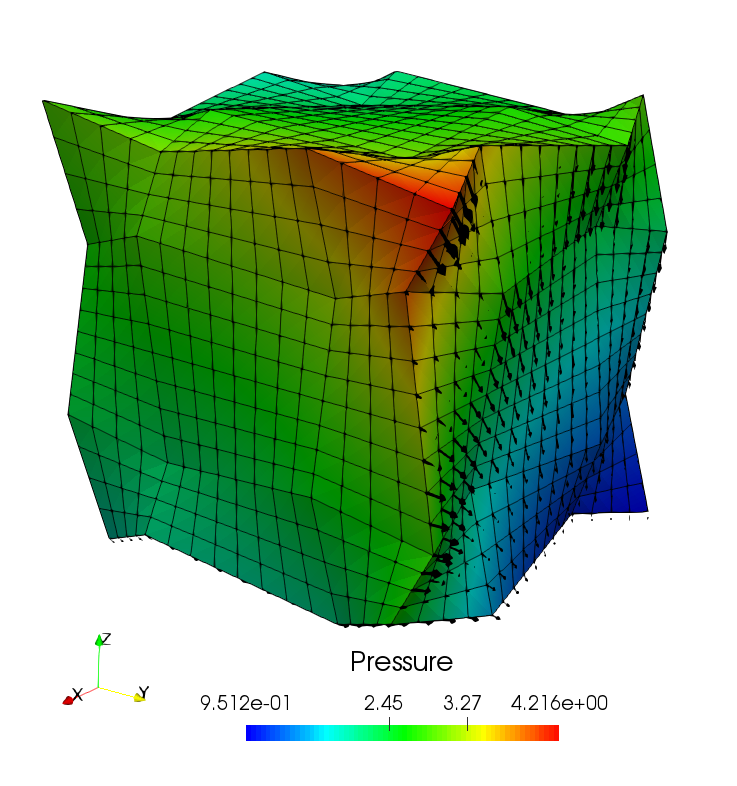}
 		\end{subfigure}
 		\begin{subfigure}[b]{0.35\textwidth}
 			\includegraphics[width=\textwidth]{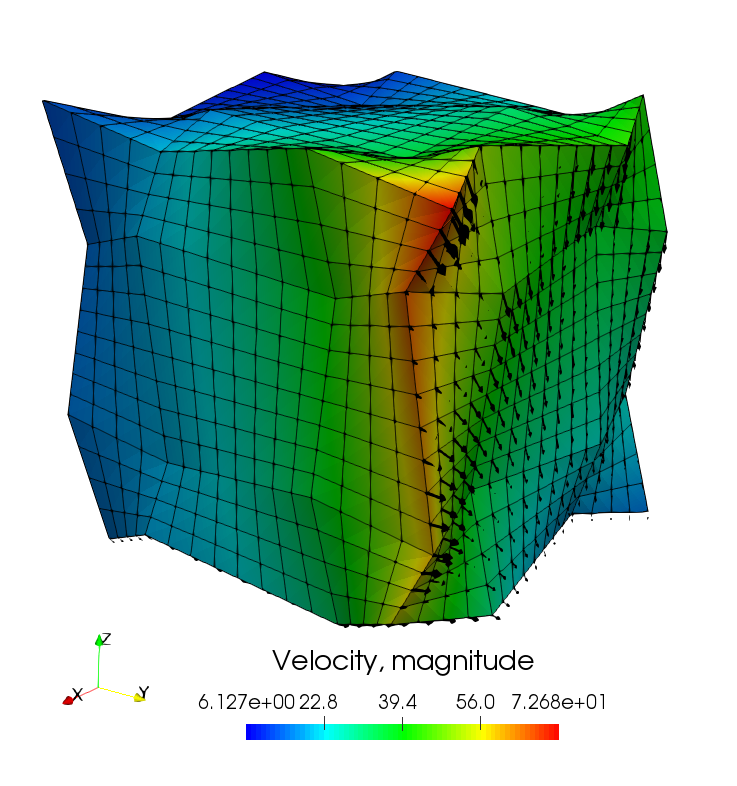}
 		\end{subfigure}
 \caption{Computed solution for Example 2 on the third level of refinement.}\label{fig:4_2}
 	\end{figure}
	%

	\begin{table}
		\begin{center}
			\begin{tabular}{r|cc|cc|cc|cc|cc} \hline
				\multicolumn{11}{c}{$k = 2$} \\ \hline
				& 
				\multicolumn{2}{c|}{$ \|\u - \u_h\| $} & 
				\multicolumn{2}{c|}{$ \|\nabla\cdot(\u - \u_h)\| $} &  
				\multicolumn{2}{c|}{$ \|p - p_h\| $} &
				\multicolumn{2}{c|}{$ |||p - p_h||| $} &
				\multicolumn{2}{c}{$ \|p - p_h^*\| $}\\ 
				$h$ & error & rate & error & rate & error & rate & error & rate & error & rate \\ \hline
				1/4	&	7.47E-03	&	--	&	2.92E-02	&	--	&	4.97E-03	&	--	&	1.63E-04	&	--	&	3.34E-04	&	--	\\	
				1/8	&	1.82E-03	&	2.0	&	7.24E-03	&	2.0	&	1.24E-03	&	2.0	&	2.23E-05	&	2.9	&	3.99E-05	&	3.1	\\	
				1/16	&	4.51E-04	&	2.0	&	1.81E-03	&	2.0	&	3.11E-04	&	2.0	&	3.07E-06	&	2.9	&	4.86E-06	&	3.0	\\	
				1/32	&	1.12E-04	&	2.0	&	4.51E-04	&	2.0	&	7.77E-05	&	2.0	&	4.12E-07	&	2.9	&	6.00E-07	&	3.0	\\	
				1/64	&	2.80E-05	&	2.0	&	1.13E-04	&	2.0	&	1.94E-05	&	2.0	&	5.38E-08	&	2.9	&	7.47E-08	&	3.0	\\	\hline\hline
				\multicolumn{11}{c}{$k = 3$} \\ \hline
				& 
				\multicolumn{2}{c|}{$ \|\u - \u_h\| $} & 
				\multicolumn{2}{c|}{$ \|\nabla\cdot(\u - \u_h)\| $} &  
				\multicolumn{2}{c|}{$ \|p - p_h\| $} &
				\multicolumn{2}{c|}{$ |||p - p_h||| $} &
				\multicolumn{2}{c}{$ \|p - p_h^*\| $}\\ 
				$h$ & error & rate & error & rate & error & rate & error & rate & error & rate \\ \hline
				1/4	&	5.06E-04	&	--	&	2.01E-03	&	--	&	2.03E-04	&	--	&	3.78E-06	&	--	&	1.23E-05	&	--	\\	
				1/8	&	6.37E-05	&	3.0	&	2.46E-04	&	3.0	&	2.54E-05	&	3.0	&	2.56E-07	&	3.9	&	6.93E-07	&	4.2	\\	
				1/16	&	7.93E-06	&	3.0	&	3.05E-05	&	3.0	&	3.17E-06	&	3.0	&	1.87E-08	&	3.8	&	4.06E-08	&	4.1	\\	
				1/32	&	9.87E-07	&	3.0	&	3.81E-06	&	3.0	&	3.97E-07	&	3.0	&	1.35E-09	&	3.8	&	2.46E-09	&	4.0	\\	
				1/64	&	1.21E-07	&	3.0	&	4.88E-07	&	3.0	&	4.96E-08	&	3.0	&	8.83E-11	&	3.9	&	1.50E-10	&	4.0	\\	\hline
			\end{tabular}
		\end{center}
		\caption{Relative errors and convergence rates for Example 2.} \label{tab:4_2}
	\end{table}

In summary, the numerical experiments confirm the theoretical
convergence results for the higher order MFMFE method both on
$h^2$-parallelograms and regular $h^2$-parallelepipeds.



\bibliographystyle{abbrv}
\bibliography{h-mfmfe}
\end{document}